\documentclass[11pt,a4paper]{amsart}
\usepackage[T1]{fontenc}
\usepackage{lmodern}
\usepackage{booktabs}
\usepackage[final]{hyperref}
\usepackage{esint, amsmath,amssymb,amsthm}
\usepackage[abbrev,backrefs]{amsrefs}
\usepackage{a4wide}
\usepackage{tocvsec2}
\usepackage{enumerate}
\usepackage{wasysym}
\usepackage{multirow}

\newtheorem{lem}{Lemma}[section]
\newtheorem{prop}{Proposition}[section]

\newtheorem{thm}{Theorem}[section]

\theoremstyle{definition}

\theoremstyle{remark}

\newtheorem{remark}{Remark}[section]
\numberwithin{equation}{section}

\newcommand{\N}{{\mathbb N}}

\newcommand{\R}{{\mathbb R}}
\newcommand{\eps}{{\varepsilon}}

\newcommand{\st}{\,:\,}
\newcommand{\dif}{\,\mathrm{d}}
\newcommand{\abs}[1]{\lvert #1 \rvert}
\newcommand{\bigabs}[1]{\bigl\lvert #1 \bigr\rvert}
\newcommand{\norm}[1]{\lVert #1 \rVert}
\newcommand{\bignorm}[1]{\bigl\lVert #1 \bigr\rVert}

\allowdisplaybreaks

\begin{document}

\title[Sharp Gagliardo-Nirenberg inequalities]{Sharp Gagliardo--Nirenberg inequalities\\ in fractional Coulomb--Sobolev spaces}
\author[J.\ Bellazzini]{Jacopo Bellazzini}
\address[J.\ Bellazzini]{Universit\`a di Sassari\\
Via Piandanna 4\\ 07100 Sassari\\ Italy}
\email{jbellazzini@uniss.it}

\author[M.\ Ghimenti]{Marco Ghimenti}
\address[M.\ Ghimenti]{Universit\`a di Pisa\\
Dipartimento di Matematica\\
Largo B. Pontecorvo 5\\
56100 Pisa\\
Italy}
\email{Marco.Ghimenti@dma.unipi.it}

\author[C.\ Mercuri]{Carlo Mercuri}
\address[C.\ Mercuri]{Swansea University\\ Department of Mathematics\\ Singleton Park\\
Swansea\\ SA2~8PP\\ Wales, United Kingdom}
\email{C.Mercuri@swansea.ac.uk}

\author[V.\ Moroz]{Vitaly Moroz}
\address[V.\ Moroz]{Swansea University\\ Department of Mathematics\\ Singleton Park\\
Swansea\\ SA2~8PP\\ Wales, United Kingdom}
\email{V.Moroz@swansea.ac.uk}

\author[J.\ Van Schaftingen]{Jean Van Schaftingen}
\address[J.\ Van Schaftingen]{Universit\'e Catholique de Louvain\\
Institut de Recherche en Math\'ematique et Phy\-sique\\
Chemin du Cyclotron 2 bte L7.01.01\\
1348 Louvain-la-Neuve \\
Belgium}
\email{Jean.VanSchaftingen@UCLouvain.be}



\date{\today}

\begin{abstract}
We prove scaling invariant Gagliardo-Nirenberg type inequalities of the form $$\|\varphi\|_{L^p(\mathbb{R}^d)}\le C\|\varphi\|_{\dot H^{s}(\mathbb{R}^d)}^{\beta} \left(\iint_{\mathbb{R}^d \times \mathbb{R}^d} \frac{|\varphi (x)|^q\,|\varphi (y)|^q}{|x - y|^{d-\alpha}} \dif x \dif y\right)^{\gamma},$$ involving fractional Sobolev norms with $s>0$ and Coulomb type energies with $0<\alpha<d$ and $q\ge 1$. We establish optimal ranges of parameters for the validity of such inequalities and discuss the existence of the optimisers. In the special case $p=\frac{2d}{d-2s}$ our results include a new refinement of the fractional Sobolev inequality by a Coulomb term. We also prove that if the radial symmetry is taken into account, then the ranges of validity of the inequalities could be extended and such a radial improvement is possible if and only if $\alpha>1$.
\end{abstract}

\maketitle

\tableofcontents

\section{Introduction and statement of results}
\settocdepth{section}

\subsection{Introduction.}
Given \(d \in \N\), \(s > 0\), \(\alpha \in (0, d)\) and \(q \in [1, \infty)\),
we define the fractional Coulomb--Sobolev space by
\[
  \mathcal{E}^{s, \alpha, q} (\R^d)
  =\Bigl\{
  \varphi : \R^d \to \R \st
  \iint_{\R^d\times \R^d}
\frac{\abs{\varphi (x)}^q\,\abs{\varphi (y)}^q}{\abs{x - y}^{d-\alpha}} \dif x \dif y <\infty
\text{ and }
\int_{\R^d} \bigabs{\abs{\xi}^{s} \widehat{\varphi} (\xi)}^2  \dif \xi < \infty
\Bigr\}.
\]
Since for every measurable function \(\varphi : \R^d \to \R\)
\begin{equation}
\label{boundSupBall}
 \Bigl(\int_{B_R(0)} \abs{\varphi}^q \dif x\Bigr)^2
 \le C R^{d - \alpha} \iint_{\R^d\times \R^d}
\frac{\abs{\varphi (x)}^q\,\abs{\varphi (y)}^q}{\abs{x - y}^{d-\alpha}} \dif x \dif y,
\end{equation}
the boundedness of the double integral on the right-hand side of \eqref{boundSupBall} ensures
that \(\varphi\) is a tempered distribution and that its Fourier transform
\(\widehat{\varphi}\) is a well-defined tempered distribution.
In particular \(\abs{\xi}^s \widehat{\varphi}\) is a well-defined distribution on
\(\R^d \setminus \{0\}\). The integrability condition in the definition of $\mathcal{E}^{s, \alpha, q} (\R^d)$ means that this distribution can be represented by an \(L^2\)--function.

In the sequel we define the fractional Laplacian \((-\Delta)^\frac{s}{2} \varphi\) by
\[
  (\widehat {(-\Delta)^\frac{s}{2} \varphi}) (\xi)
  = \bigl(2 \pi \abs{\xi}^2\bigr)^\frac{s}{2} \widehat{\varphi} (\xi).
\]
We endow the space \(\mathcal{E}^{s, \alpha, q} (\R^d)\) with the norm
\[
  \norm{\varphi}_{\mathcal{E}^{s, \alpha, q} (\R^d)}
  =\Biggl(\bignorm{(-\Delta)^\frac{s}{2} \varphi}^2_{L^2 (\R^d)}
  +\biggl(\ \iint_{\R^d\times \R^d}
\frac{\abs{\varphi (x)}^q\,\abs{\varphi (y)}^q}{\abs{x - y}^{d-\alpha}} \dif x \dif y\biggr)^{\frac{1}{q}}
\Biggr)^{\frac{1}{2}}.
\]
In particular, when \(s < \frac{d}{2}\), a function \(\varphi\) is in the space
\(\mathcal{E}^{s, \alpha, q} (\R^d)\) if and only if \(\varphi \in \dot{H}^s (\R^d)\)
and
\[
  \iint_{\R^d\times \R^d}
\frac{\abs{\varphi (x)}^q\,\abs{\varphi (y)}^q}{\abs{x - y}^{d-\alpha}} \dif x \dif y < \infty.
\]
Following the arguments in \cite{MMVS}*{Section 2}, the space
\(\mathcal{E}^{s, \alpha, q} (\R^d)\) is a Banach space (see Proposition~\ref{propositionComplete} below).

The space \(\mathcal{E}^{s, \alpha, q} (\R^d)\) is the natural domain for the fractional Coulomb--Dirichlet type energy
\begin{equation*}\label{e-CD-energy}
  \bignorm{(-\Delta)^\frac{s}{2} \varphi}^2_{L^2 (\R^d)}+\iint_{\R^d\times \R^d}
\frac{\abs{\varphi (x)}^q\,\abs{\varphi (y)}^q}{\abs{x - y}^{d-\alpha}} \dif x \dif y,
\end{equation*}
which appears in models of mathematical physics related to multi-particle systems.
Typically, the Coulomb term with $q=2$ represents the electrostatic repulsion between the particles.
Relevant models include Thomas--Fermi--Dirac--von\thinspace{}Weizs\"acker (TFDW) models of Density Functional theory \citelist{\cite{Lieb81}\cite{BenguriaBrezisLieb}\cite{LeBris-Lions-2005}}; or Schr\"odinger--Poisson--Slater approximation to Hartree--Fock theory \cite{Catto-2013}.
Nonquadratic ($q\neq 2$) Coulombic energies appear in a possible zero mass limit of the relativistic Thomas--Fermi--von\thinspace{}Weizsacker (TFW) energy, see \citelist{\cite{BPO-2002}\cite{BLS-2008}} where $d=3$, $s=1$, $\alpha=2$, $q=3$; or \cite{BGT-2012}*{Section 2} where $d=2$, $s=1$, $\alpha=1$, $q=4$.
The fractional case $s=1/2$ occurs in the ultra-relativistic models, cf. \citelist{\cite{Lieb-Yau}\cite{LiebSeiringer2010}}.
In particular, $d=2$, $s=1/2$ and $\alpha=1$ appears in the recent TFDW theory of charge screening in graphene \cite{LMM-2015}, where relevant powers are $q=2$ or $q=1$.
Interpolation inequalities \eqref{e-main} associated with the space \(\mathcal{E}^{s, 2s, 2} (\R^d)\) are in some cases equivalent to the Lieb--Thirring type inequalities \cite{Lundholm-2016}*{Theorem 3}, which are fundamental
in the study of stability of non-relativistic ($s=1$) and ultra-relativistic ($s=1/2$) matter \cite{LiebSeiringer2010}.

Mathematically, the space $\mathcal E^{1,2,2}(\R^3)$ has been introduced and studied by P.-L.\thinspace{}Lions \citelist{\cite{Lions-1981}*{Lemma 4}\cite{Lions1987}*{(55)}} and in D.\thinspace{}Ruiz \cite{Ruiz-ARMA}*{section 2}. In particular, P.-L.\thinspace{}Lions established a Coulomb-Sobolev interpolation inequality
\begin{equation}\label{E122}
\|\varphi\|_{L^3(\R^3)}\le C\|\nabla \varphi\|_{L^2(\R^3)}^{1/2}\biggl(\ \iint_{\R^3 \times \R^3}
\frac{\abs{\varphi (x)}^2\,\abs{\varphi (y)}^2}{\abs{x - y}} \dif x \dif y\biggl)^{1/2},
\end{equation}
which holds for all $\varphi \in \mathcal E^{1,2,2}(\R^3)$.
Lions' proof relies on the quadratic structure of the nonlocal term ($q=2$) and the special relation $\alpha=2s$ and cannot be extended beyond these restrictions.
Coulomb--Sobolev inequalities in the fractional space \(\mathcal{E}^{s, \alpha, 2} (\R^d)\) had been studied in \citelist{\cite{BOV}\cite{BFV}} using methods of fractional calculus, while the non-quadratic case \(\mathcal{E}^{1, \alpha, q} (\R^d)\) had been introduced and studied in \cite{MMVS} using Morrey type estimates.

We emphasize that unlike the classical Hardy--Littlewood--Sobolev inequality,
Coulomb--Sobolev inequality is a {\em lower} bound on the nonlocal Coulomb energy.
In particular, \eqref{E122} ensures the continuous embedding $\mathcal E^{1,2,2}(\R^3)\subset L^3(\R^3)\cap L^6(\R^3)$.
D.\thinspace{}Ruiz in \cite{Ruiz-ARMA}*{Theorem 1.2} observed that if the radial symmetry is taken into account,
then the ranges of validity of the Coulomb--Sobolev inequalities could be extended. As a consequence,
he established an improved embedding $\mathcal E_{\mathrm{rad}}^{1,2,2}(\R^3)\subset L^p(\R^3)\cap L^6(\R^3)$, for any $p>18/7$.
In \cite{MMVS} the radial improvement was extended to \(\mathcal{E}_{\mathrm{rad}}^{1, \alpha, q} (\R^d)\) with any $\alpha>1$.
It was also shown that no radial improvement occurs when $\alpha\le 1$.
In \cite{BGO}, the radial improvement was obtained in \(\mathcal{E}_{\mathrm{rad}}^{s, 2, 2} (\R^3)\) for $1/2<s<3/2$.
The result however did not include the physically important ultra-relativistic case $s=1/2$.
Technically, this was related to the failure of pointwise Strauss type estimates on the radial functions in fractional Sobolev spaces of order $s\le 1/2$.

The aim of the present paper is threefold:
\begin{itemize}
\item
We extend Coulomb--Sobolev inequalities associated to the space \(\mathcal{E}^{s, \alpha, q} (\R^d)\) to arbitrary $s>0$ and $q\ge 1$,
thus completing the studies in  \cite{BFV} ($q=2$) and \cite{MMVS} ($s=1$). Our proof is different from the proofs in \citelist{\cite{BFV}\cite{MMVS}}.
It is based only on the standard fractional Gagliardo-Nirenberg inequality and a fractional chain rule.
\item
We analyze a family of refined Sobolev inequalities, which appear as a special end point case of the interpolation inequalities in $\mathcal{E}^{s, \alpha, \frac{d+\alpha}{d-2s}}(\R^d)$. For some values of parameters we establish the existence of optimizers to the refined Sobolev inequalities.
The existence of the optimisers is new even in the previously studied case $s=1$.
\item
We obtain a radial improvement of Coulomb-Sobolev inequalities in the space \(\mathcal{E}_{\mathrm{rad}}^{s, \alpha, q} (\R^d)\) of radially symmetric functions
for the complete range $s>0$, $q\ge 1$, $\alpha>1$. This includes, in particular, previously open case $s\le 1/2$.
We also show that a radial improvement is possible if and only if $\alpha>1$, so $\alpha=1$ is a universal critical constant which does not depend on any other parameter.
In addition, we observe that $q=\big(\frac{2}{1-2s}\big)_+$ plays a special role as the only value of $q$ where the radial embedding interval is closed.
\end{itemize}
All of our results are essentially sharp, which is demonstrated by a range of counterexamples confirming optimality.

\subsection{Coulomb--Sobolev inequalities.}
Our first main result in this paper is the continuous embedding
$$\mathcal{E}^{s, \alpha, q} (\R^d) \hookrightarrow L^\frac{2(2q s+\alpha)}{2s+\alpha}(\R^d).$$
More specifically, we establish a family of scaling--invariant interpolation inequalities for the space $\mathcal{E}^{s, \alpha, p}(\R^d)$.

\begin{thm}[Coulomb--Sobolev inequalities]\label{maint}
Let $d\in\N$, $s > 0$, $0<\alpha<d$, $q,p\in[1,\infty)$ and $q(d-2s)\neq d+\alpha$.
There exists a constant $C=C(d,s,\alpha,q,p)>0$ such that the scaling invariant inequality
\begin{equation}\label{e-main}
  \Vert \varphi \Vert _{p}
 \le C\Vert \varphi\Vert _{\dot H^{s}(\R^d)}^{\frac{p(d+\alpha)-2dq}{p(d+\alpha-q(d-2s))}}
    \biggl(\ \iint_{\R^d \times \R^d}
\frac{\abs{\varphi (x)}^q\,\abs{\varphi (y)}^q}{\abs{x - y}^{d-\alpha}} \dif x \dif y\biggl)^{\frac{2d-p(d-2s)}{2p(d+\alpha-q(d-2s))}}
\end{equation}
holds for every function $\varphi\in \mathcal{E}^{s, \alpha, q}(\R^d)$ if and only if
\begin{align}
  p &\ge \frac{2(2qs+\alpha)}{2s+\alpha}
    && \text{if } s \ge \frac{d}{2},\label{e-main-0}\\
  p&\in\Big[ \frac{2(2 q s+\alpha)}{2s+\alpha}, \frac{2d}{d-2s}\Big]
    &&\text{if } s < \frac{d}{2} \quad \text{ and } \quad \frac{1}{q}>\frac{d-2s}{d+\alpha},\label{e-main-1}\\
  p&\in\Big[\frac{2d}{d-2s}, \frac{2(2qs +\alpha)}{2s+\alpha}\Big]
    &&\text{if } s < \frac{d}{2} \quad \text{ and } \quad  \frac{1}{q}<\frac{d-2s}{d+\alpha}.\label{e-main-2}
\end{align}
Moreover, if $p$ is not an end--point of the intervals \eqref{e-main-0}--\eqref{e-main-2}, i.e. $p\neq \frac{2(2qs+\alpha)}{2s+\alpha}$ and $p\neq \frac{2d}{d-2s}$,
then the best constant  for \eqref{e-main} is achieved.
\end{thm}

In the case $s=1$ inequality \eqref{e-main} was known for \(d = 3\), \(\alpha = 2\) and \(q = 2\) \citelist{\cite{Lions1987}*{(55)}\cite{Ruiz-ARMA}*{Theorem 1.5}};
and for $d\in\N$, $\alpha\in(0,N)$ and $q\ge 1$ \cite{MMVS}*{Theorem 1}.
The fractional inequality \eqref{e-main} first appeared for $d=3$, $s=1/2$, $\alpha=2$ and $q=2$ in \cite{BOV}*{Proposition 2.1};
and for $d\in\N$, $s>0$, $\alpha\in(0,d)$ and $q=2$ in \cite{BFV}*{Proposition 2.1}.

\subsection{Refined Sobolev inequalities.}
The special case $q(d-2s)= d+\alpha$, which corresponds to
$p=\frac{2d}{d-2s}$ and $q=\frac{d+\alpha}{d-2s}$, is not covered by
the previous theorem and the exponents in \eqref{e-main} are meaningless. In this special case we
obtain a refinement of the Sobolev embedding, extending the one
observed for $s=1$ \cite{MMVS}*{(1.7)} and for $q=2$
\cite{BFV}*{Proposition 2.1}.

\begin{thm}[Endpoint refined Sobolev inequality]\label{thm:sobimp}
Let $d\in\N$, $0<s<\frac{d}{2}$, $0<\alpha<d$.
Then there exists $C=C(d, s, \alpha)>0$ such that the inequality
\begin{equation}\label{eq:Sobimp}
\Vert \varphi \Vert _{L^\frac{2d }{d-2s}(\R^d)} \leq  C \Vert \varphi\Vert _{\dot H^{s}(\R^d)}^{\frac{\alpha(d-2s)}{d(2s+\alpha)}} \left(\iint_{\R^d \times \R^d}
\frac{\vert \varphi(x)\vert^\frac{d+\alpha}{d-2s} \vert \varphi(y)\vert^\frac{d+\alpha}{d-2s}}{\vert x-y\vert^{d-\alpha}} \dif x \dif y\right)^{\frac{s(d-2s)}{d(2 s +\alpha)}}
\end{equation}
holds for all $\varphi\in \mathcal{E}^{s, \alpha, \frac{d+\alpha}{d-2s}}(\R^d)$.
\end{thm}

\begin{remark}
It is interesting to compare our refinement for Sobolev embedding with two other improvements.
The G\'erard--Meyer--Oru improvement \citelist{\cite{BCD}*{Theorem 1.43}\cite{Ledoux-2003}} states that
if $0<s<\frac{d}{2}$ and $\theta\in\mathcal S(\R^d)$ is such that $\widehat\theta$ has compact support, has value $1$ near the origin and satisfies $0\leq\widehat\theta\leq 1$, then
\begin{equation}
\label{ineqGerardMeyerOru}
\Vert \varphi \Vert _{L^\frac{2d}{d-2s}(\R^d)}\leq C(d,s,\theta) \Vert  \varphi \Vert _{\dot H^s(\R^d)}^{1 -  \frac{2s}{d}} \left( \sup_{\lambda >0} \lambda^{\frac{d}{2}+s} \Vert \theta(\lambda \,\cdot) \star \varphi\Vert _\infty \right)^{\frac{2s}{d}}\quad\forall \varphi \in \dot H^s(\R^d).
\end{equation}
The Palatucci--Pisante improvement \cite{PP}*{Theorem 1.1} (see also \cite{VanSchaftingen2014}*{(4.2)}) states that if $0<s<\frac{d}{2}$, then
\begin{equation}
\label{eqineqPalPis}
\Vert \varphi\Vert _{L^\frac{2d}{d-2s}(\R^d)}\leq C(d,s) \Vert \varphi\Vert _{\dot H^s(\R^d)}^{1 - \frac{2 s}{d}} \Vert \varphi\Vert _{\mathcal{M}^{1, \frac{d}{2} - s}}^{\frac{2 s}{d}}\qquad\forall \varphi\in \dot H^s(\R^d).
\end{equation}
In the last inequality, the Morrey norm is defined as
$$
  \Vert \varphi \Vert _{\mathcal{M}^{r,\gamma}}:=\sup_{R >0,\, x\in \R^d} R^{\gamma} \Bigl(\fint_{B_R(x)}\vert u\vert^r\Bigr)^\frac{1}{r} ;
$$
one proof of \eqref{eqineqPalPis} relies on \eqref{ineqGerardMeyerOru} and on the observation that
$$
\lambda^{\frac{d}{2}+s} \Vert \theta(\lambda \,\cdot) \star \varphi\Vert _\infty
  \le C \Vert \varphi\Vert _{\mathcal{M}^{1, \frac{d-2s}{2}}}.
$$
In our case we have by H\"older's inequality and monotonicity of the integral
$$
  \biggl(R^{\frac{d}{2} - s} \fint_{B_R(x)}\vert \varphi\vert \biggr)^{\frac{d+\alpha}{d-2s}}
  \leq
R^{\frac{d+\alpha}{2}}  \fint_{B_R(x)}\vert \varphi\vert^{\frac{d+\alpha}{d-2s}} \leq C \left(\ \iint_{\R^d \times \R^d}
\frac{\vert \varphi(x)\vert^{\frac{d+\alpha}{d-2s}} \vert \varphi(y)\vert^{\frac{d+\alpha}{d-2s}}}{\vert x-y\vert^{d-\alpha}} \dif x\dif y\right)^\frac{1}{2}
$$
so that it is clear that Coulomb norm controls the Morrey norm $\mathcal{M}^{1, \frac{d}{2} - s}$.
On the other hand, the exponent $\frac{\alpha(d-2s)}{d(2s+\alpha)} = (1 - \frac{2 s}{d})\frac{1}{1 + 2s/\alpha}$ for $\dot H^s$-norm in our improvement
is always less than the exponent $1 - \frac{2 s}{d}$ for $\dot H^s$-norm in \eqref{ineqGerardMeyerOru} and \eqref{eqineqPalPis}.
This suggests that the inequality \eqref{eq:Sobimp} cannot be derived directly from the already known ones.
\end{remark}

\begin{remark}
The refinement of the Sobolev inequality in Theorem~\ref{thm:sobimp} is sharp.
Indeed, by scaling it can be proved that if a scaling invariant inequality of the following form holds
\begin{equation}\label{eq:Sobbg}
\Vert \varphi \Vert _{L^\frac{2d }{d-2s}(\R^d)} \leq  C(d, s, \alpha) \Vert \varphi\Vert _{\dot H^{s}(\R^d)}^{\beta} \left(\iint_{\R^d \times \R^d}
\frac{\vert \varphi(x)\vert^{\frac{d+\alpha}{d-2s}} \vert \varphi(y)\vert^{\frac{d+\alpha}{d-2s}}}{\vert x-y\vert^{d-\alpha}} \dif x \dif y\right)^{\gamma},
\end{equation}
then the exponents $\gamma$ and $\beta$ are related by the equation
\begin{equation*}\label{eq:condexp}
\frac{d-2s}{2}=\Bigl(\frac{d}{2}-s\Bigr)\beta+(d+\alpha)\gamma
\end{equation*}
On the other hand, estimates \eqref{eq:stimq-m-nonrad}--\eqref{eq:stimcoul-m-nonrad} in the proof of Theorem~\ref{maint} below imply that
\begin{equation*}\label{eq:condexp2}
\frac{d-2s}{2d}\le\frac{\beta}{2}+\gamma.
\end{equation*}
We conclude that $\beta\ge\frac{\alpha(d-2s)}{d(2s+\alpha)}$ is necessary for \eqref{eq:Sobbg} to hold.
\end{remark}

Interpolating between the refined and classical Sobolev inequalities,
we obtain a new family of interpolation inequalities, for which the best constant is achieved.

\begin{thm}[Non-endpoint refined Sobolev inequalities]\label{thm:sobimp-eps}
Let $d\in\N$, $0<s<\frac{d}{2}$, $0<\alpha<d$ and $0<\eps<\frac{s(d-2s)}{d(2s +\alpha)}$.
Then there exists $C=C(d, s, \alpha, \eps)>0$ such that the inequality
\begin{equation}\label{eq:Sobimp-eps}
||\varphi ||_{\frac{2d }{d-2s}} \leq  C\|\varphi\|_{\dot H^{s}(\R^d)}^{\frac{\alpha(d-2s)}{2sd+\alpha d}+\eps\frac{2(\alpha+d)}{d-2s}} \left(\iint_{\R^d \times \R^d}
\frac{|\varphi(x)|^{\frac{d+\alpha}{d-2s}} |\varphi(y)|^\frac{d+\alpha}{d-2s}}{|x-y|^{d-\alpha}} \, \dif x \dif y\right)^{\frac{s(d-2s)}{d(2s +\alpha)}-\eps}
\end{equation}
holds for all $\varphi\in \mathcal{E}^{s, \alpha, \frac{d+\alpha}{d-2s}}(\R^d)$.
Moreover, the best constant  for \eqref{eq:Sobimp-eps} is achieved.
\end{thm}

When $\eps=\frac{s(d-2s)}{d(2s +\alpha)}$ the inequality \eqref{eq:Sobimp-eps} is the classical Sobolev inequality.

The existence of optimizers for the non-endpoint inequality \eqref{eq:Sobimp} provides a partial answer towards the question raised
in the case $s=1$ in \cite{MMVS}*{Section 1.5.5}.
The existence of optimizers for the endpoint inequality \eqref{eq:Sobimp} remains open.

\subsection{Radial improvements.}
We now consider the question of embeddings for \emph{radial functions}.
Since the symmetric decreasing rearrangement \emph{increases} the nonlinear nonlocal
Coulomb energy term, the situation might be more favorable for radial functions.
Our next result shows that for the subspace of radially symmetric functions in the Coulomb--Sobolev space $\mathcal{E}^{s, \alpha, q}_{\mathrm{rad}}(\R^d)$
the intervals \eqref{e-main-0}--\eqref{e-main-2} of the validity of the Coulomb--Sobolev inequality \eqref{e-main}
can be extended provided that $\alpha>1$.

\begin{thm}[Sharp Improvement in the radial case for $\alpha>1$]\label{thm:radialpartial}
Let $d\ge 2$, $s>0$, $1<\alpha<d$, $q,p\in[1,\infty)$, $q(d-2s)\neq d+\alpha$  and
\begin{equation*}\label{e-qRad}
p_{\mathrm{rad}}:=
q+\frac{\bigl((2s-1)q+2\bigr)(d-\alpha)}{2s(d+\alpha-2)+d-\alpha}.
\end{equation*}
There exists a constant $C_{\mathrm{rad}}=C_{\mathrm{rad}}(d,s,\alpha,q,p)>0$ such that the scaling invariant inequality
\begin{equation}\label{e-main-radial}
\Vert \varphi \Vert _{L^p(\R^d)}
  \le C_{\mathrm{rad}}\Vert \varphi\Vert _{\dot H^{s}(\R^d)}^{\frac{p(d+\alpha)-2dq}{p(d+\alpha-q(d-2s))}}
  \Biggl(\ \iint_{\R^d \times \R^d}
\frac{\abs{\varphi (x)}^q\,\abs{\varphi (y)}^q}{\abs{x - y}^{d-\alpha}} \dif x \dif y\Biggr)^{\frac{2d-p(d-2s)}{2p(d+\alpha-q(d-2s))}}
\end{equation}
hold for all radially symmetric functions $\varphi\in \mathcal{E}^{s, \alpha, q}_{\mathrm{rad}}(\R^d)$ if and only if
\begin{align}
  p &> p_{\mathrm{rad}}
    && \text{if } s \ge \frac{d}{2},\label{e-main-radial-0}\\
  p&\in\Big(p_{\mathrm{rad}}, \frac{2d}{d-2s}\Big]
    &&\text{if } s < \frac{d}{2} \;\text{ and }\; \frac{1}{q}>\frac{d-2s}{d+\alpha},\label{e-main-radial-1}\\
  p&\in\Big[\frac{2d}{d-2s}, p_{\mathrm{rad}}\Big)
    &&\text{if } s < \frac{d}{2} \;\text{ and }\; \frac{1}{q}<\frac{d-2s}{d+\alpha}, \;\; \frac{1}{q}\ne\frac{1-2s}{2},\label{e-main-radial-2}\\
  p&\in\Big[\frac{2d}{d-2s}, q\Big]
  &&\text{if } s < \frac{1}{2} \;\text{ and }\;  \frac{1}{q}=\frac{1-2s}{2}.\label{e-main-radial-2plus}
  \end{align}
If $0<\alpha\le 1$ then inequality \eqref{e-main-radial} holds on $\mathcal{E}^{s, \alpha, q}_{\mathrm{rad}}(\R^d)$ if and only if \eqref{e-main}
holds on $\mathcal{E}^{s, \alpha, q}(\R^d)$.
\end{thm}

In the important special case $s=1/2$ we have the simplified expression $p_{\mathrm{rad}}=q+\frac{d-\alpha}{d-1}$,
while for $s=0$ we formally obtain $p_{\mathrm{rad}}=2$.

In the special case \(d = 3\), $s=1$, \(\alpha = 2\) and \(q = 2\) the improved radial inequality \eqref{e-main-radial} was first established in \cite{Ruiz-ARMA}*{Theorem 1.2}.
For $d\in\N$, $s=1$, $\alpha\in(0,d)$ and $q\ge 1$ the improved radial inequalities \eqref{e-main} were studied in \cite{MMVS}*{Theorem 4}.
The fractional case $d=3$, $1/2<s<3/2$, $\alpha=2$, $q=2$ was considered in \cite{BGO}.

We shall emphasise that the radial improvement is possible for any $s>0$ but if and only if $\alpha>1$.
The \emph{universality} of the threshold $\alpha=1$ which does not depend on any other parameter in the problem is quite interesting.

Another new and purely fractional phenomenon is the special role of the exponent $q=\frac{2}{1-2s}$ in the case $s<1/2$.
Observe that for $s\ge 1/2$ we always have $p_{\mathrm{rad}}>q$, while $p_{\mathrm{rad}}<q$ if $s<1/2$ and $q>\frac{2}{1-2s}$,
the latter requires $q>\frac{d+\alpha}{d-2s}$. If $s<1/2$ and $q=\frac{2}{1-2s}$ then $p_{\mathrm{rad}}=q$ and this is the only case when the endpoint embedding $\mathcal{E}^{s, \alpha, q}_{\mathrm{rad}}(\R^d)\hookrightarrow L^{p_{\mathrm{rad}}}(\R^d)$ is valid.

Finally, we prove that the embedding $\mathcal{E}^{s, \alpha, q}_{\mathrm{rad}}(\R^d)\hookrightarrow L^p(\R^d)$
is compact provided that $p$ is not an endpoint of the embedding intervals.

\begin{thm}[Compact embeddings for radial functions]\label{thm:radial comp}
Let $d\ge 2$, $s>0$, and $q\in[1,\infty)$.  Moreover we assume that $p$ is away from the endpoints of the intervals in (\ref{e-main-0})--(\ref{e-main-2}) when $0<\alpha \leq 1$ and in (\ref{e-main-radial-0})--(\ref{e-main-radial-2plus}) when $1<\alpha<d$.
Then, the embedding $ \mathcal{E}^{s, \alpha, q}_{\mathrm{rad}}(\R^d) \hookrightarrow L^{p}(\R^d)$ is compact.
\end{thm}

Compactness of the radial embedding implies in a standard way the existence of radial optimizers associated to the inequalities \eqref{e-main-radial},
cf. \cite{MMVS}*{Section 7} where the case $s=1$ was considered.

\subsection{Open questions.}
Here we list some of the open problems related to the results in the present work.

\subsubsection{Radial symmetry breaking.}
It is an open question whether the optimal constants $C$ and $C_{\mathrm{rad}}$ in \eqref{e-main} and \eqref{e-main-radial} share the same value for $p$ in the intervals \eqref{e-main-0}--\eqref{e-main-2}, where both constants are well-defined.
A result by D.\thinspace{}Ruiz \cite{Ruiz-ARMA}*{theorem 1.7} gives an indirect indication that $C<C_{\mathrm{rad}}$ might be possible, at least for the values of $p$ close to $\frac{2(2qs+\alpha)}{2s+\alpha}$. However the problem remains open even in the well--studied case $s=1$, $\alpha=2$, $q=2$.

\subsubsection{Radial compactness in the borderline case $\alpha=1$ and $p=\frac{2(2qs+\alpha)}{2s+\alpha}$.}
Compactness of the borderline embedding $\mathcal{E}^{s,1, q}_{\mathrm{rad}}(\R^d) \hookrightarrow L^{\frac{2(2qs+1)}{2s+1}}(\R^d)$ is open. This includes $\mathcal{E}^{1/2,1,2}_{\mathrm{rad}}(\R^2) \hookrightarrow L^3(\R^2)$, which appears in the ultra-relativistic TFDW model for graphene studied in \cite{LMM-2015}.

\subsubsection{Other symmetries.}
We believe that the critical threshold $\alpha=1$ for the radial improvement is related to the essential uni-dimensionality of radial functions. It seems plausible that the Coulomb-Sobolev embeddings can be improved for other types of symmetries. A natural conjecture would be that the relevant value of the critical constant $\alpha$ is the number of variables on which the symmetric functions depend. For example, for axisymmetric functions in $\R^3$, we would expect a critical value $\alpha = 2$.

\subsection{Outline.}
The rest of the paper is organised as follows.
Section~\ref{s-complete} contains a short proof of the completeness of the Coulomb--Sobolev spaces.
In Section~\ref{s-nonradial} we discuss the spaces \(\mathcal{E}^{s, \alpha, q} (\R^d)\) in the nonradial context
and show that interpolation inequalities of Theorems~\ref{maint} and~\ref{thm:sobimp} can be deduced from the standard fractional Gagliardo--Nirenberg inequality \eqref{eq:GN} using a fractional chain rule.
We also discuss the existence of the optimisers and prove Theorem~\ref{thm:sobimp-eps}.
In Section~\ref{s-radial} we derive the radial improvement of Theorem~\ref{thm:radialpartial} as a consequence of
Ruiz's inequality for Coulomb energy (see Theorem~\ref{thm:ruiz}) and de Napoli's interpolation inequality (see Theorem~\ref{thm:denapoli}),
which is a fractional extension of the classical pointwise Strauss type bounds valid only for $s>1/2$.
In case $s\le 1/2$ we replace de Napoli's pointwise bounds by Rubin's inequality (Theorem~\ref{thm:Rubin}), which is a refinement for radial functions of the classical Stein--Weiss inequality.
In Section~\ref{s-radial-opt} we construct special families of functions which are used to prove the optimality of the radial embeddings,
while in Section~\ref{compactness section} we prove the compactness of the radial embedding.

\subsection{Asymptotic notation.}

For real valued functions $f(t), g(t) \geq 0$, we write: \smallskip

$f(t)\lesssim g(t)$ if there exists $C>0$ independent of $t$
such that $f(t) \le C g(t)$;

$f(t)\simeq g(t)$ if $f(t)\lesssim g(t)$ and
$g(t)\lesssim f(t)$.

\noindent
As usual, $C,c,c_1$, etc., denote generic positive constants independent of $t$.

\section{Completeness of the fractional Coulomb--Sobolev space}\label{s-complete}

As in \cite{MMVS}*{Section 2}, it is not difficult to see that the space \(\mathcal{E}^{s, \alpha, q} (\R^d)\) is a normed space.

\begin{prop}\label{propositionComplete}
For every $d\in\N$, $s>0$, $0<\alpha<d$ and $q\in[1,\infty)$, the normed space \(\mathcal{E}^{s, \alpha, q} (\R^d)\) is complete.
\end{prop}
\begin{proof}
If \((u_n)_{n \in \N}\) is a Cauchy sequence in \(\mathcal{E}^{s, \alpha, q}(\R^d)\),
then \(((-\Delta)^\frac{s}{2} u_n)_{n \in \N}\) is a Cauchy sequence in \(L^2 (\R^d)\) and
there exists thus \(f \in L^2 (\R^d)\) such that \(((-\Delta)^\frac{s}{2} u_n)_{n \in \N}\)
converges strongly to \(f\) in \(L^2 (\R^d)\).
On the other hand, by \eqref{boundSupBall} we have for every \(R > 0\),
\[
 \lim_{m, n \to \infty} \int_{B_R (0)} \abs{u_n - u_m}^q  = 0.
\]
There exists thus a measurable function \(u : \R^d \to \R\)
such that \((u_n)_{n \in \N}\) converges
to \(u\) in \(L^q_{\mathrm{loc}} (\R^d)\).
By Fatou's lemma, we have
\begin{multline*}
 \lim_{n \to \infty} \iint_{\R^d \times \R^d} \frac{\abs{u_n (x) - u (x)}^q\,\abs{u_n (y) - u_n (y)}^q}
 {\abs{x - y}^{d - \alpha}}\dif x \dif y\\
 \le \lim_{n \to \infty} \liminf_{m \to \infty} \iint_{\R^d \times \R^d}\frac{\abs{u_n (x) - u_m (x)}^q\,\abs{u_n (y) - u_m (y)}^q}
 {\abs{x - y}^{d - \alpha}}\dif x \dif y.
\end{multline*}

It remains now to prove that \((-\Delta)^\frac{s}{2} u = f\).
We observe that by \eqref{boundSupBall},
\[
 \lim_{n \to \infty} \sup_{R > 0} \frac{1}{R^\frac{d - \alpha}{2}} \int_{B_R(0)} \abs{u_n - u}^q = 0.
\]
Therefore \((u_n)_{n \in \N}\) converges to \(u\) as tempered distributions \(\R^d\),
and thus the sequence \((\widehat{u}_n)_{n \in \N}\) converges to \(\widehat{u}\) as tempered distributions on \(\R^d\).
It follows that \(((2 \pi)^{s/2}\abs{\xi}^s \widehat{u_n})_{n \in \N}\) converges to \(2 \pi^{s/2} \abs{\xi}^s \abs{\xi}^s \widehat{u}\) as distributions on \(\R^d\).
Since on the other hand, \(((2 \pi)^{s/2}\abs{\xi}^s \widehat{u_n})_{n \in \N}\) converges to \(\widehat{f}\) it follows that \((-\Delta)^\frac{s}{2} u = f\).
\end{proof}

\section{Gagliardo--Nirenberg inequalities: Proof of Theorems~\ref{maint},~\ref{thm:sobimp} and~\ref{thm:sobimp-eps}}\label{s-nonradial}
We first establish the endpoint inequality.
\begin{thm}\label{lem:gen}
Let $d\in\N$, $s>0$, $0<\alpha<d$ and $q\in[1,\infty)$.
Then the following inequality holds
\[
  \norm{\varphi}_{L^{\frac{2 (2qs+\alpha)}{2s+\alpha}}(\R^d)}
  \lesssim  \bignorm{(-\Delta)^\frac{s}{2}\varphi}_{L^2 (\R^d)}^{\frac{\alpha}{2q s+\alpha}} \biggl(\iint_{\R^d \times \R^d}
\frac{\abs{\varphi (x)}^q\,\abs{\varphi (y)}^q}{\abs{x - y}^{d-\alpha}} \dif x \dif y\biggr)^\frac{s}{2q s+\alpha}
\qquad\forall\varphi\in \mathcal{E}^{s, \alpha, q}(\R^d).
\]
In particular, $\mathcal{E}^{s, \alpha, q} (\R^d) \hookrightarrow L^\frac{2(2q s+\alpha)}{2s+\alpha}(\R^d)$ continuously.
\end{thm}
The above inequality in the particular case $q=1$ implies that $\mathcal{E}^{s, \alpha, 1}(\R^d)$ embeds continuously into $H^s(\R^d)$.

\begin{proof}[Proof of Theorem~\ref{lem:gen}]
Recall that for all $\phi\in L^1_{\textrm{loc}}(\R^d)$ such that
\begin{equation}\label{finite}
\iint_{\R^d \times \R^d}
\frac{\phi(x)\phi(y)}{\abs{x - y}^{d-\alpha}} \dif x \dif y <\infty
\end{equation}
it holds that
\begin{equation}\label{RieszId}
  \bignorm{(-\Delta)^{-{\frac{\alpha}{4}}}\phi}_{L^2(\R^d)}^2
  =c  \iint_{\R^d \times \R^d}
\frac{\phi(x)\phi(y)}{\abs{x - y}^{d-\alpha}} \dif x \dif y.
\end{equation}
Moreover we recall the endpoint Gagliardo--Nirenberg inequality (see for example \cite{BCD}*{Theorem 2.44})
\begin{equation}\label{eq:GN}
\bignorm{(-\Delta)^{\frac{\alpha}{4}} \psi}_{L^p(\R^d)}\leq C \Vert \psi\Vert _{L^2(\R^d)}^{\frac{2s }{\alpha + 2 s}}
\bignorm{(-\Delta)^{\frac{\alpha}{4}+\frac{s}{2}} \psi}_{L^r(\R^d)}^\frac{\alpha}{\alpha + 2s} \,
\end{equation}
where
\begin{equation*}\label{eq:rel2}
\frac{1}{p}=\frac{1}{2}\Bigl(\frac{2s}{\alpha+2s}\Bigr) + \frac{1}{r}\Bigl(\frac{\alpha}{\alpha+2s}\Bigr).
\end{equation*}
When $q=1$ by (\ref{finite}) and (\ref{RieszId}) it holds that
\begin{equation}\label{q1norm}
\bignorm{(-\Delta)^{-{\frac{\alpha}{4}}}\varphi}_{L^2(\R^d)}^2
  =c  \iint_{\R^d \times \R^d}
\frac{\varphi(x)\varphi(y)}{\abs{x - y}^{d-\alpha}} \dif x \dif y \leq  c\iint_{\R^d \times \R^d}
\frac{\abs{\varphi (x)}\,\abs{\varphi (y)}}{\abs{x - y}^{d-\alpha}} \dif x \dif y.
\end{equation}
Setting $\psi=(-\Delta)^{-{\frac{\alpha}{4}}}\varphi$ and $p=r=2,$ (\ref{eq:GN}) together with (\ref{q1norm}) yields the inequality for $q=1$.

Let $q>1$. Setting $\psi=(-\Delta)^{-{\frac{\alpha}{4}}}|\varphi|^q$ in \eqref{eq:GN}, we get
\begin{equation*}\label{eq:GN2}
\bignorm{\abs{\varphi}^{q}}_{L^p(\R^d)}
  \leq C \bignorm{(-\Delta)^{-{\frac{\alpha}{4}}} \abs{\varphi}^{q}}_{L^2(\R^d)}^{\frac{2 s}{\alpha + 2s}}
         \bignorm{(-\Delta)^\frac{s}{2} \abs{\varphi}^{q}}_{L^r(\R^d)}^{\frac{\alpha}{\alpha + 2s}} \,
\end{equation*}
which implies
\begin{equation}\label{eq:GN3}
\Vert \varphi\Vert _{L^{q p}(\R^d)}^{q}
  \leq C \bignorm{(-\Delta)^{-{\frac{\alpha}{4}}} \abs{\varphi}^{q}}_{L^2(\R^d)}^{\frac{2 s}{\alpha + 2s}}
  \bignorm{(-\Delta)^\frac{s}{2} \varphi}_{L^2(\R^d)}^{\frac{\alpha}{\alpha + 2 s}} \Vert  \varphi\Vert _{L^{(q-1)l}(\R^d)}^{(q-1)\frac{\alpha}{\alpha + 2s}}
\end{equation}
by the fractional chain rule where $\frac{1}{r}=\frac{1}{2} +\frac{1}{l}$ \cite{Gatto-2002}*{Corollary of Theorem 5}.
Now choosing $l$ such that $(q-1)l=q p$, i.e. such that
$$\frac{1}{l}=\frac{q-1}{q p}$$
we conclude that $p=\frac{2\alpha+4q s}{q(2s+\alpha)}$.
By \eqref{eq:GN3} and setting $\phi=|\varphi|^q$ in (\ref{RieszId}), this
concludes the proof.
\end{proof}

\begin{proof}[Proof of Theorem~\ref{maint} and Theorem~\ref{thm:sobimp}]
The exponents for the refined Sobolev inequality  given by Theorem~\ref{thm:sobimp} are derived directly from the endpoint Gagliardo--Nirenberg inequality of Theorem~\ref{lem:gen}.

The scaling-invariant inequalities of Theorem~\ref{maint} follows from the fact that by interpolation between Theorem~\ref{lem:gen} and the classical fractional Sobolev embedding,
$\mathcal{E}^{s, \alpha, q} (\R^d) \hookrightarrow L^p(\R^d)$ continuously for
\begin{align*}
& p\in\left( \frac{2 (2qs+\alpha)}{2s+\alpha},  \frac{2d}{d-2s}\right] & & \text{if}\ 1< q<\frac{d+\alpha}{d-2s} . \\
& p\in\left[ \frac{2d}{d-2s} , \frac{2(2qs+ \alpha)}{2s+\alpha}  \right) & & \text{if}\ q>\frac{d+\alpha}{d-2s}
\end{align*}
Indeed, let us consider the scaling $u_{\lambda}(x)=\lambda^{\frac{d}{p}}u(\lambda x)$ such that $\Vert u_{\lambda}\Vert _{L^p(\R^d)}=\vert \abs{u}\vert _{L^p(\R^d)}$.
From the embedding we get
\[
\Vert u_{\lambda}\Vert _{L^p(\R^d)}^2
  \lesssim \bignorm{(-\Delta)^\frac{s}{2} u_{\lambda}}_{\dot L^2 (\R^d)}^2+\biggl(\iint_{\R^d\times \R^d}
\frac{\abs{u_{\lambda}(x)}^q\, \abs{u_{\lambda}(y)}^q}{\abs{x - y}^{d-\alpha}} \dif x \dif y\biggr)^{\frac{1}{q}},
\]
which gives by scaling
\begin{equation}\label{eq:scc}
\norm{u}_{L^p(\R^d)}^2\lesssim  \lambda^{\frac{2d}{p}-d+2s}\bignorm{(-\Delta)^\frac{s}{2}u_{\lambda}}_{L^2 (\R^d)}^2+\lambda^{\frac{2d}{p}-\frac{(d+\alpha)}{q}}\biggl(\iint_{\R^d\times \R^d}
\frac{\abs{u_{\lambda}(x)}^q\, \abs{u_{\lambda}(y)}^q}{\abs{x - y}^{d-\alpha}} \dif x \dif y\biggr)^{\frac{1}{q}}.
\end{equation}
Notice that when $q=\frac{d+\alpha}{d-2s}$
and $p=\frac{2d}{d-2s}$ we obtain as expected $\frac{2d}{p}-d+2s=0$, $\frac{2d}{p}-\frac{(d+\alpha)}{q}=0$.
Minimizing the right-hand side  of \eqref{eq:scc} with respect to $\lambda $ we get the scaling invariant inequality given by Theorem~\ref{maint}.
The same computation of course works in the radial case.
\smallskip

\noindent
\emph{Optimality of the embedding intervals.}
Given a nonnegative function $\eta\in C^\infty_c(\R^d)\setminus\{0\}$ and a vector \(a \in \R^d\setminus\{0\}\), for $k\in\N$ set
$$u_{a,k}(x)=\eta(x + k a).$$
Following \cite{Ruiz-ARMA}*{Section 5}, we define the functions \(v_{a,m} \in C^\infty_c (\R^N)\) by
\[
  v_{a,m} = \sum_{k = 1}^m u_{a,k}.
\]
Then for $|a|\to\infty$ we obtain
\begin{gather}
 \Vert v_{a,m}\Vert _{L^p(\R^d)}^{p}\simeq m,\label{eq:stimq-m-nonrad}\\
 \Vert v_{a,m}\Vert _{\dot H^{s}(\R^d)}^{2}  \lesssim  m, \label{eq:stimsob-m-nonrad}\\
 \iint_{\R^d\times \R^d}\frac{\vert v_{a,m}(x)\vert^q\,\vert v_{a,m}(y)\vert^{q}}{\abs{x - y}^{d-\alpha}} \dif x \dif y \lesssim  m   \label{eq:stimcoul-m-nonrad}.
\end{gather}
To deduce \eqref{eq:stimsob-m-nonrad}, choose $k\in\N$ such that $k\ge s$.
Interpolating between homogeneous $L^2$ and $\dot{H}^{k}$ norms (cf. \cite{BCD}*{Proposition 1.32}),
for $|a|\to\infty$ we conclude that
\[
 \Vert v_{a,m}\Vert _{\dot H^{s}(\R^d)}^{2}
 \le \Vert v_{a,m}\Vert_{\dot H^{k}(\R^d)}^{\frac{2s}{k}}\Vert v_{a,m}\Vert _{L^2(\R^d)}^{2\big(1-\frac{s}{k}\big)}\lesssim \Big(m\|\eta\|_{\dot H^{k}(\R^d)}^2\Big)^{\frac{s}{k}}\Big(m\|\eta\|_{L^2(\R^d)}^2\Big)^{1-\frac{s}{k}}\lesssim m.
\]

Using the diagonal argument, from \eqref{e-main} we deduce that for all sufficiently large \(m \in \N\) it must hold
\[
  m\lesssim m^{\frac{p(d+\alpha)-2dq}{2(d+\alpha-q(d-2s))}} m^{\frac{2d-p(d-2s)}{2(d+\alpha-q(d-2s))}},
\]
which implies the optimality of the embedding intervals \eqref{e-main-0}--\eqref{e-main-2}.

\smallskip
\noindent
\emph{Existence of the optimizers.}
The existence of optimizers follows almost identically to the proof of \cite{BFV}*{Theorem 2.2}, see also \cite{BGO}*{proof of Corollary 0.1}.
We only sketch the argument.

Fix $p$ inside one of the intervals \eqref{e-main-0}--\eqref{e-main-2}.
By homogeneity and scaling we can assume that an optimizing sequence $(\varphi_n)_{n\in\N}$ in $\mathcal E^{s,\alpha,q}(\R^d)$ satisfies
$$\|\varphi_n\|_{\dot H^s(\R^d)}=\iint_{\R^d \times \R^d}
\frac{|\varphi_n(x)|^q |\varphi_n(y)|^q}{|x-y|^{d-\alpha}} \, \dif x \dif y=1,$$
and $\|\varphi_n\|_{L^{p}(\R^d)}=C(d,s,\alpha)+o(1)$.

Since $p$ is not an end-point of the intervals \eqref{e-main-0}--\eqref{e-main-2},
we can use interpolation inequality \eqref{e-main} to find a uniform upper bound on $\|\varphi_n\|_{L^{p_1}(\R^d)}$ and $\|\varphi_n\|_{L^{p_2}(\R^d)}$,
for some $p_1 < p < p_2$. Therefore, via the $pqr$-lemma \cite{FrolichLiebLoss}*{Lemma 2.1 p.258}
and Lieb's compactness lemma in $\dot H^s(\R^d)$ \cite{BFV}*{Lemma 2.1}, we conclude that
$\varphi_n\rightharpoonup\bar\varphi\neq 0$ in $H^s(\R^d)$.
Finally, using the non-local Brezis--Lieb splitting lemma for the Coulomb term \cite{MMVS}*{Proposition 4.8},
the existence of a maximizer could be proved similarly to the arguments in \cite{BFV}*{pp.661--662} (see also the proof of Theorem~\ref{thm:sobimp-eps}
below for similar estimates).
\end{proof}

\begin{proof}[Proof of Theorem~\ref{thm:sobimp-eps}]
Inequality \eqref{eq:Sobimp-eps} is obtained directly by interpolation
between the classical Sobolev inequality and endpoint refined Sobolev inequality \eqref{eq:Sobimp}
\smallskip

To prove that the best constant $C(d, s, \alpha, \eps)$ in \eqref{eq:Sobimp-eps} is achieved, we will use the following result.

\begin{thm}[Gerard--Meyer--Oru]\label{gmo}
Let $0<s<d/2$ and let $\theta\in\mathcal S(\R^d)$ be such that $\hat\theta$ has compact support, has value $1$ near the origin and satisfies $0\leq\hat\theta\leq 1$. Then there is a constant $C=C_{s,d}(\theta)$ such that for all $u\in \dot H^s(\R^d)$,
$$
\|u\|_{\frac{2d}{d-2s}}\leq C \| u\|_{\dot H^s}^{\frac{d-2s}{d}} \left( \sup_{A>0} A^{d/2+s} \|\theta(A\,\cdot) \star u\|_\infty \right)^{\frac{2s}{d}}.
$$
\end{thm}

Consider a maximizing sequence $(\varphi_n)_{n\in\N}$ for \eqref{eq:Sobimp-eps} such that
$\ \|\varphi_n\|_{\dot H^{s}(\R^d)}=1$ and
$$||\varphi_n||_{\frac{2d }{d-2s}}=\left(C(d, s, \alpha, \eps)+o(1)\right) \big(D(\varphi_n)\big)^{\frac{s(d-2s)}{d(2 s +\alpha)}-\eps},$$
where for brevity, we denoted
$$D(\varphi):=\iint_{\R^d \times \R^d}
\frac{|\varphi(x)|^{\frac{d+\alpha}{d-2s}} |\varphi(y)|^{\frac{d+\alpha}{d-2s}}}{|x-y|^{d-\alpha}} \, \dif x \dif y.$$
Using the endpoint refined Sobolev inequality we infer that
\begin{equation*}
\big(D(\varphi_n)\big)^{\frac{s(d-2s)}{d(2s+\alpha)}-\eps} \lesssim ||\varphi_n||_{\frac{2d }{d-2s}}
\lesssim \big(D(\varphi_n)\big)^{\frac{s(d-2s)}{d(2s +\alpha)}}. \nonumber
\end{equation*}
This implies that
$$1\lesssim   D(\varphi_n)$$
and hence,
\begin{equation}\label{eq:asyim}
1\lesssim ||\varphi_n ||_{\frac{2d }{d-2s}}.
\end{equation}
Let $\bar \varphi$ denotes the weak limit of $(\varphi_n)$ in $\dot H^s(\R^d)$.
Recall that our inequality  \eqref{eq:Sobimp-eps} is critical, i.e. it is both scaling and translation invariant.
From Theorem~\ref{gmo} together with \eqref{eq:asyim} there exists sequences $(x_n)_{n \in \N}\) in $\R^d$ of translations and $(A_n)_{n \in \N}\) in \(\R^{+}$ of dilations such that
$$\inf_n A_n^{\frac{d}{2}+s} \int_{\R^d} \theta (A_n(x_n-y))\varphi_n(y) dy >0.$$
This fact implies by the change of variable that
$$
  A_n^{s-\frac{d}{2}}\varphi_n\Bigl(\frac{x-x_n}{A_n}\Bigr)\rightharpoonup \bar \varphi\neq 0.
$$

The fact  that $\bar \varphi$ is an optimizer is now standard.
By the Brezis--Lieb type splitting properties \cite{Brezis-Lieb-1983} of the three terms in \eqref{eq:Sobimp-eps}
(for the splitting of the nonlocal term $D$ see \cite{MMVS}*{Proposition 4.7}), we obtain
\begin{multline*}
C(d, s, \alpha, \eps)^{-\frac{2d}{d-2s}}\left(||\bar \varphi||_{\frac{2d }{d-2s}}^{\frac{2d}{d-2s}}+||\varphi_n -\bar \varphi ||_{\frac{2d }{d-2s}}^{\frac{2d}{d-2s}}+o(1)\right)\\ \nonumber
\geq \left(  \| \bar \varphi\|_{\dot H^{s}(\R^d)}^2+ \|\varphi_n-\bar \varphi \|_{\dot H^{s}(\R^d)}^2+o(1)\right)^{\frac{d\alpha}{d(2s+\alpha)}+\eps \frac{2d(\alpha+d)}{(d-2s)^2}}\big(D(\bar \varphi) +D(\varphi_n-\bar \varphi)\big)^{\frac{2ds}{d(2s +\alpha)}-\eps \frac{2d}{d-2s}}.
\end{multline*}
Since
\[
 \Bigl(\frac{d\alpha}{d(2s +\alpha)}+\eps \frac{2d(\alpha+d)}{(d-2s)^2}\Bigr)
 + \Bigl(\frac{2ds}{d(2s +\alpha)}-\eps \frac{2d}{d-2s} \Bigr)
 = 1+ \eps \frac{2d (\alpha + 2s)}{(d - 2 s)^2}> 1,
\]
As a consequence of the discrete H\"older inequality we have
\begin{multline*}
a^{\frac{2d\alpha}{d(2s +\alpha)}+\eps \frac{4d(\alpha+d)}{(d-2s)^2}} c^{\frac{2ds}{d(2s +\alpha)}-\eps \frac{2d}{d-2s}} + b^{\frac{2d\alpha}{d(2s +\alpha)}+\eps \frac{4d(\alpha+d)}{(d-2s)^2}} e^{\frac{2ds}{d(2s +\alpha)}-\eps \frac{2d}{d-2s}}\\
\le \left( a^2 + b^2 \right)^{\frac{d\alpha}{d(2s +\alpha)}+\eps \frac{2d(\alpha+d)}{(d-2s)^2}} \left(c+e \right)^{\frac{2ds}{d(2s +\alpha)}-\eps \frac{2d}{d-2s}}
\end{multline*}
for all $a,b,c,e\geq 0$. Hence
\begin{multline*}
C(d, s, \alpha, \eps)^{-\frac{2d}{d-2s}}\left(||\bar \varphi||_{\frac{2d }{d-2s}}^{\frac{2d}{d-2s}}+||\varphi_n -\bar \varphi ||_{\frac{2d }{d-2s}}^{\frac{2d}{d-2s}}+o(1)\right)\\
\geq   \|\bar \varphi\|_{\dot H^{s}(\R^d)}^{\frac{2d\alpha}{d(2s +\alpha)}+\eps \frac{4d(\alpha+d)}{(d-2s)^2}} D(\bar \varphi)^{\frac{2ds}{d(2s +\alpha)}-\eps \frac{2d}{d-2s}}\\
+ \|\varphi_n - \bar \varphi\|_{\dot H^{s}(\R^d)}^{\frac{2d\alpha}{d(2s +\alpha)}+\eps \frac{4d(\alpha+d)}{(d-2s)^2}} D(\varphi_n- \bar \varphi)^{\frac{2ds}{d(2s +\alpha)}-\eps \frac{2d}{d-2s}}+o(1).
\end{multline*}
Therefore we can conclude that
$$C(d, s, \alpha, \eps)^{-\frac{2d}{d-2s}}||\bar \varphi||_{\frac{2d }{d-2s}}^{\frac{2d}{d-2s}}\geq  \|\bar \varphi\|_{\dot H^{s}(\R^d)}^{\frac{2d\alpha}{d(2s +\alpha)}+\eps \frac{4d(\alpha+d)}{(d-2s)^2}} D(\bar \varphi)^{\frac{2ds}{d(2s +\alpha)}-\eps \frac{2d}{d-2s}}+o(1),$$
which implies  that $\varphi$ is an optimizer.
\end{proof}

\section{Sharp improvement in the radial case}\label{s-radial}

In order to establish the radial inequality \eqref{e-main-radial} we will use a version of the weighted estimate involving the Coulomb term which was originally established by Ruiz \cite{Ruiz-ARMA}.

\begin{thm}[Ruiz \cite{Ruiz-ARMA}*{Theorem 1.1}, see also \cite{MMVS}*{Proposition 3.8}]\label{thm:ruiz}
Let $d\in\N$, $0<\alpha<d$, $q\in[1,\infty)$.
Then for every $\eps>0$ and $R>0$ there exists $C=C(d,\alpha,q,\eps)>0$ such that
for all $\varphi\in L^\frac{2dq}{d+\alpha}(\R^d)$,
\begin{equation}\label{Ruiz-ext}
  \int_{\R^d \setminus B_R(0)} \frac{\abs{\varphi(x)}^q}{\abs{x}^{\frac{d - \alpha}{2}+\eps}}\dif x
  \le \frac{C}{R^{\eps}}
    \biggl(\ \iint_{\R^d \times \R^d}\frac{\vert \varphi(x)\vert^{q}\,\vert \varphi(y)\vert^{q}}{\vert x-y\vert^{d-\alpha}}
    \dif x\dif y\biggr)^\frac{1}{2},
\end{equation}
\begin{equation}\label{Ruiz-int}
  \int_{B_R(0)} \frac{\abs{\varphi(x)}^q}{\abs{x}^{\frac{d - \alpha}{2}-\eps}}\dif x
  \le CR^\eps \biggl(\ \iint_{\R^d \times \R^d}
    \frac{\vert \varphi(x)\vert^{q}\,\vert \varphi(y)\vert^{q}}{\vert x-y\vert^{d-\alpha}}
    \dif x\dif y\biggr)^\frac{1}{2}.
\end{equation}
\end{thm}

We will also employ two different estimate on the functions in $\dot H^{s}_{\mathrm{rad}}(\R^d)$.
In the case $s>1/2$ our proof of \eqref{e-main-radial} relies on the following interpolation result.

\begin{thm}[De N\'{a}poli  \cite{DeNapoli-symmetry}*{Theorem 3.1}]\label{thm:denapoli}
Let $d\geq 2$, $s>\frac{1}{2}$, $r>1$ and
\begin{equation}\label{DeNapoli-1}
-(d-1)\le a<d(r-1).
\end{equation}
Then
\begin{equation}\label{DeNapoli-ineq}
\abs{\varphi (x)}\leq C(d,s,r,a)\abs{x}^{-\sigma} \Vert (-\Delta)^\frac{s}{2} \varphi\Vert _{L^2(\R^d)}^{\theta}\Vert \varphi\Vert _{L^r_a(\R^d)}^{1-\theta}
\qquad\forall\varphi\in\dot H^{s}_{\mathrm{rad}}(\R^d)\cap L^r_a(\R^d),
\end{equation}
where $\sigma=\frac{2s(d-1)+(2s-1)a}{(2s-1)r+2}$, $\theta=\frac{2}{(2s-1)r+2}$ and $ L^r_a(\R^d)$ is the weighted Lebesgue space with the norm
$$\Vert u\Vert _{L^r_a(\R^d)}=\left(\int_{\R^d} \abs{x}^a\abs{u(x)}^r\dif x\right)^{\frac{1}{r}}.$$
\end{thm}

\begin{remark}
The inequality \eqref{thm:denapoli} has important special cases:
\begin{enumerate}[i)]
\item
When $r=\frac{2d}{d-2s}$ and $a=0$ we obtain Cho--Ozawa's inequality \cite{CO}:
\begin{equation}\label{Strauss}
\sup_{\vert x\vert >0}\abs{\varphi (x)}\lesssim\vert x\vert^{-\frac{d-2s}{2}}\Vert \varphi\Vert _{\dot{H}^{s}(\R^d)}
\qquad\forall\varphi\in \dot{H}_{\mathrm{rad}}^{s}(\R^d),
\end{equation}

\item
When $r=2$ and $a=0$ we obtain Ni type inequality
\begin{equation*}\label{Ni}
\sup_{\vert x\vert >0}\abs{\varphi (x)}\lesssim\vert x\vert^{-\frac{d-1}{2}}\Vert \varphi\Vert _{\dot{H}^{s}(\R^d)}^\frac{1}{2s}\Vert \varphi\Vert _{L^2(\R^d)}^{1-\frac{1}{2s}}
\qquad\forall\varphi\in \dot{H}_{\mathrm{rad}}^{s}(\R^d).
\end{equation*}
\end{enumerate}

\end{remark}

In the case $s\le 1/2$ pointwise estimates on functions in $\dot{H}_{\mathrm{rad}}^{s}(\R^d)$ are no longer available.
Instead, our proof of \eqref{e-main-radial} relies on the radial version of the classical Stein--Weiss estimate \cite{SteinWeiss1958}.

\begin{thm}[Rubin  \citelist{\cite{Rubin-1982}\cite{DeNapoli-2011}*{Theorem 1.2}\cite{DeNapoli-2014}}]\label{thm:Rubin}
Let $d\ge 2$ and $0<s<d/2$. Then
\begin{equation}\label{Rubin}
\left(\int_{\R^d} \abs{\varphi (x)}^r\abs{x}^{-\beta r}\dif x\right)^\frac{1}{r}\le C(d,s,r,\beta)\Vert \varphi\Vert _{\dot{H}^{s}(\R^d)}
\qquad\forall\varphi\in \dot{H}_{\mathrm{rad}}^{s}(\R^d),
\end{equation}
where $r\ge 2$ and
\begin{equation}\label{Rubin-1}
-(d-1)\Big(\frac{1}{2}-\frac{1}{r}\Big)\le\beta<\frac{d}{r},
\end{equation}
\begin{equation}\label{Rubin-2}
\frac{1}{r}=\frac{1}{2}+\frac{\beta-s}{d}.
\end{equation}
\end{thm}

\begin{remark}
The difference with the classical (non-radial) Stein--Weiss estimate \cite{SteinWeiss1958}
is only in the extended range \eqref{Rubin-1} for $\beta$
(in the non-radial case we must have $0\le\beta<\frac{d}{r}$).
Note special cases of \eqref{Rubin}:
\begin{enumerate}[i)]
\item When $\beta=s$ and $s<\frac{d}{2}$ we obtain $r=2$ which gives the Hardy inequality:
\begin{equation*}\label{DeNapoli-Hardy}
\left(\int_{\R^d} \abs{\varphi(x)}^2\abs{x}^{-2s}\dif x\right)^\frac{1}{2}\lesssim \Vert \varphi\Vert _{\dot{H}^{s}(\R^d)}\qquad\forall \varphi\in \dot{H}_{\mathrm{rad}}^{s}(\R^d),
\end{equation*}

\item  When $\beta=0$ and $s<\frac{d}{2}$ we obtain $r=\frac{2d}{d-2s}$ which gives the Sobolev estimate:
\begin{equation*}\label{DeNapoli-Sobolev}
\left(\int_{\R^d} \abs{\varphi}^\frac{2d}{d-2s}\right)^{\frac{1}{2}-\frac{s}{d}}\lesssim \Vert \varphi\Vert _{\dot{H}^{s}(\R^d)}\qquad\forall \varphi\in \dot{H}_{\mathrm{rad}}^{s}(\R^d),
\end{equation*}

\item  When $\beta=-(d-1)\big(\frac{1}{2}-\frac{1}{r}\big)$ and $s<\frac{1}{2}$ we see from \eqref{Rubin-2} that $r=\frac{2}{1-2s}$ and hence $\beta=-(d-1)s$,
so we obtain a ``limiting'' inequality
\begin{equation*}\label{DeNapoli-radial}
\left(\int_{\R^d} \abs{\varphi}^\frac{2}{1-2s}\abs{x}^\frac{2s(d-1)}{1-2s}\dif x\right)^{\frac{1}{2}-s}\lesssim \Vert \varphi\Vert _{\dot{H}^{s}(\R^d)}\qquad\forall \varphi\in \dot{H}_{\mathrm{rad}}^{s}(\R^d).
\end{equation*}
\end{enumerate}

\end{remark}

A corollary of Rubin's inequality is an integral replacement of the Cho--Ozawa bound \eqref{Strauss}.

\begin{lem}[Weak Ni's inequality]\label{cstep0}
Let $d\ge 2$, $0<s\le 1/2$ and $\frac{1}{2}-s\leq \frac{1}{p}\leq \frac{1}{2}-\frac{s}{d}$.
Then for $R>0$,
\begin{equation}\label{e-weakNI}
\int_{\R^d\setminus B_R(0)} \vert \varphi\vert^p\leq C(d,s,p) R^{d-p\left(\frac{d}{2}-s\right)} \Vert \varphi\Vert _{\dot{H}_{\textrm{rad}}^s(\R^d)}^p
\qquad\forall \varphi\in \dot{H}_{\mathrm{rad}}^{s}(\R^d).
\end{equation}
\end{lem}

\begin{proof}
Follows from Rubin's inequality \eqref{Rubin} by setting $r=p$ and $\beta= \frac{2d-p(d-2s)}{2p}$.
\end{proof}

Using \eqref{Ruiz-ext}, \eqref{DeNapoli-ineq} and \eqref{Rubin} in the exterior
and the classical Sobolev inequality in the interior of a ball we deduce the following.

\begin{prop}\label{prop-radial-ext}
Let $d\ge 2$, $s>0$, $1<\alpha<d$ and $\big(\frac{d-2s}{d+\alpha}\big)_+<\frac{1}{q}\le 1$.
Then the space $\mathcal{E}^{s, \alpha, q}_{\mathrm{rad}}(\R^d)$ is continuously embedded into $L^p(\R^d)$ for
\begin{align}
p\in\Big(p_{\mathrm{rad}},\frac{2d}{d-2s}\Big]&\quad\text{and}\quad s<\frac{d}{2},\\
p>p_{\mathrm{rad}}&\quad\text{and}\quad s\ge\frac{d}{2}.
\end{align}
\end{prop}

\begin{proof}
It is sufficient to establish continuous embedding $\mathcal{E}^{s, \alpha, q}_{\mathrm{rad}}(\R^d)\hookrightarrow L^p(\R^d)$
only for $p$ in a small \emph{right} neighbourhood of $p_{\mathrm{rad}}$, the remaining values of $p$ are then covered by interpolation via Theorem~\ref{lem:gen}.
Given $R>0$, we shall estimate the $L^p$--norm of a function $\varphi \in \mathcal{E}^{s, \alpha, q}_{\mathrm{rad}}(\R^d)$ separately in the interior and exterior of the ball $B_R(0)$.
Since $p<\frac{2d}{d-2s}$, in the interior of the ball $B_R(0)$ we estimate by Sobolev inequality
\begin{equation*}\label{e-int}
\int_{B_R(0)} \vert \varphi\vert^p\le CR^{1-p\left(\frac{1}{2}-\frac{s}{d}\right)}\Vert \varphi\Vert _{\dot{H}^{s}(\R^d)}^{p}.
\end{equation*}

The estimate in the exterior of the ball $B_R(0)$ will be split into the cases $s> 1/2$ and $s\le 1/2$.
Observe that $p>p_{\mathrm{rad}}>q$, since $q<\frac{d+\alpha}{d-2s}$.
For a small $\eps>0$, denote
$$\gamma:=\frac{d-\alpha}{2}+\eps.$$

\smallskip\noindent
\emph{Case $s>1/2$.}
Using successively the inequalities \eqref{DeNapoli-ineq}, \eqref{Ruiz-ext} and \eqref{Strauss}, we estimate
\begin{multline}\label{e-int-larges}
\int_{\R^d\setminus B_R(0)} \vert \varphi\vert^p
\le
\sup_{\vert x\vert >R}\Big(\vert \varphi(x)\vert \,\vert x\vert^\frac{\gamma}{p-q}\Big)^{p-q}
\int_{\R^d\setminus B_R(0)}\frac{\vert \varphi (x)\vert^{q}}{\vert x\vert^\gamma}\dif x
\\
\lesssim
\Vert \varphi\Vert _{\dot{H}^{s}(\R^d)}^{\theta(p-q)}
\left(\int_{\R^d}\frac{\vert \varphi (x)\vert^{q}}{\vert x\vert^\gamma}\dif x\right)^{\frac{(1-\theta)(p-q)}{q}}
\int_{\R^d\setminus B_R(0)}\frac{\vert \varphi (x)\vert^{q}}{\vert x\vert^\gamma}\dif x
\\
\lesssim
\Vert \varphi\Vert _{\dot{H}^{s}(\R^d)}^{\theta(p-q)}
\biggl(\frac{1}{R^{2\eps}}\iint_{\R^N}\frac{\vert \varphi(x)\vert^{q}\,\vert \varphi(y)\vert^{q}}{\vert x-y\vert^{d-\alpha}}\dif x\dif y\biggr)^{\frac{1}{2}+\frac{(1-\theta)(p-q)}{2q}}
\\
\qquad+
\Vert \varphi\Vert _{\dot{H}^{s}(\R^d)}^{p-q}
\biggl(\frac{1}{R^{2\eps}}\iint_{\R^N}\frac{\vert \varphi(x)\vert^{q}\,\vert \varphi(y)\vert^{q}}{\vert x-y\vert^{d-\alpha}}\dif x\dif y\biggr)^{\frac{1}{2}}
\biggl(\int_{B_R(0)}\vert x\vert^{-\frac{d-2s}{2}q-\gamma}\dif x\biggr)^{\frac{(1-\theta)(p-q)}{q}},
\end{multline}
where $\theta=\frac{2}{(2s-1)q+2}$. The application of \eqref{DeNapoli-ineq} requires that
\begin{equation}\label{e-deNap-prad}
\frac{\gamma}{p-q}\le \sigma=\frac{2s(d-1-\gamma)+\gamma}{(2s-1)q+2},
\end{equation}
which is fulfilled for a sufficiently small $\eps>0$ if $p>p_{\mathrm{rad}}$.
The last integral in \eqref{e-int-larges} is finite when
\begin{equation}\label{e-deNap-q}
-\frac{d-2s}{2}q-\gamma<-d;
\end{equation}
this is the case for a sufficiently small $\eps>0$ when $q<\frac{d+\alpha}{d-2s}$.

\smallskip\noindent
\emph{Case $s\le 1/2$.}
Let $r>p>q$ and $\theta\in[0,1]$ be such that $\frac{\theta}{q}+\frac{1-\theta}{r}=\frac{1}{p}$, i.e.
$\theta=\frac{q}{p}\frac{r-p}{r-q}$.
By the H\"older inequality together with \eqref{Ruiz-ext} and \eqref{Rubin}, we estimate
\begin{equation}\label{e-ext-smalls}
\begin{split}
\int_{\R^d\setminus B_R(0)} \vert \varphi\vert^p&\le
\biggl(\int_{\R^d\setminus B_R(0)}\vert \varphi(x)\vert^{r}\vert x\vert^{\gamma\frac{r-p}{p-q}}\dif x\biggr)^\frac{p-q}{r-q}
\biggl(\int_{\R^d\setminus B_R(0)}\frac{\vert \varphi(x)\vert^{q}}{\vert x\vert^\gamma}\dif x\biggr)^\frac{r-p}{r-q}
\\
&\lesssim
\biggl(\int_{\R^d\setminus B_R(0)}\vert \varphi(x)\vert^{r}\vert x\vert^{-r\beta}\dif x\biggr)^\frac{p-q}{r-q}
\biggl(\frac{1}{R^{2\eps}}\iint_{\R^N}\frac{\vert \varphi(x)\vert^{q}\,\vert \varphi(y)\vert^{q}}{\vert x-y\vert^{d-\alpha}}\dif x\dif y\biggr)^{\frac{1}{2}\frac{r-p}{r-q}}\\
&\lesssim
\Vert \varphi\Vert _{\dot{H}^{s}(\R^d)}^{r\frac{p-q}{r-q}}
\biggl(\frac{1}{R^{2\eps}}\iint_{\R^N}\frac{\vert \varphi(x)\vert^{q}\,\vert \varphi(y)\vert^{q}}{\vert x-y\vert^{d-\alpha}}\dif x\dif y\biggr)^{\frac{1}{2}\frac{r-p}{r-q}},
\end{split}
\end{equation}
where in view of \eqref{Rubin-2} we must express $r$ and $\beta$ as
\begin{equation*}\label{e-sys-res}
r=\frac{2(\gamma p-d(p-q))}{2\gamma-(d-2s)(p-q)},
\qquad \beta=\frac{1}{2}\frac{\gamma(2d-p(d-2s))}{\gamma p-d(p-q)}.
\end{equation*}
Note that $\beta<0$ for sufficiently small $\eps>0$, since $q<\frac{d+\alpha}{d-2s}$ and $p<\frac{2d}{d-2s}$.
Hence \eqref{Rubin-1} requires
\begin{equation*}
\beta\ge-\frac{d-1}{2}\frac{\gamma(p-2)-2s(p-q)}{\gamma p-d(p-q)}.
\end{equation*}
The latter is satisfied provided that
\begin{equation}\label{e-pgamma}
p\ge p_\eps:=2\frac{qs(d-1)+\gamma}{2s(d-1)+\gamma(1-2s)},
\end{equation}
where $p_\eps\searrow p_{\mathrm{rad}}$ as $\eps\to 0$.
In addition, observe that $r\nearrow\frac{2}{1-2s}$ as $p=p_\eps$ and $\eps\to 0$,
which in particular, ensures that we can choose $r>p$ and $r>2$ in \eqref{Rubin}.
We conclude that \eqref{e-ext-smalls} holds for $p>p_{\mathrm{rad}}$, provided that $\eps>0$ is sufficiently small.
\end{proof}

\begin{prop}\label{prop-radial-int}
Let $d\ge 2$, $0<s<\frac{d}{2}$, $1<\alpha<d$ and $\frac{d+\alpha}{d-2s}<q<\infty$.
Then the space $\mathcal{E}^{s, \alpha, q}_{\mathrm{rad}}(\R^d)$ is continuously embedded into $L^p(\R^d)$ for
\begin{align}
p\in\Big[\frac{2d}{d-2s},p_{\mathrm{rad}}\Big)&\quad\text{and}\quad \frac{1}{q}\neq \frac{1}{2}-s,\\
p\in\Big[\frac{2d}{d-2s},p_{\mathrm{rad}}\Big]&\quad\text{and}\quad \frac{1}{q}=\frac{1}{2}-s.
\end{align}
\end{prop}

\begin{proof}
Note that for $\frac1q\neq\frac12-s$ it is sufficient to establish continuous embedding $\mathcal{E}^{s, \alpha, q}_{\mathrm{rad}}(\R^d)\hookrightarrow L^p(\R^d)$
only for $p$ in a small \emph{left} neighbourhood of $p_{\mathrm{rad}}$, the remaining values of $p$ are then covered by interpolation via Theorem~\ref{lem:gen}.

Given $R>0$, we shall estimate the $L^p$--norm of a function $\varphi \in \mathcal{E}^{s, \alpha, q}_{\mathrm{rad}}(\R^d)$ separately in the interior and exterior of the ball $B_R(0)$.
The proof will be splitted into a number of separate cases, which we outline in Table~\ref{table1}.

\begin{table}[h!]
\begin{center}
\begin{tabular}{c c  c  c }
\toprule
$s$ & $q$ & $B_R(0)$ & $\R^d\setminus B_R(0)$  \\
\midrule[\heavyrulewidth]
$s>1/2$ &$q>\frac{d+\alpha}{d-2s}$ & De Napoli + Ruiz as in \eqref {e-int-larges} & Sobolev + Cho-Ozawa \eqref{Strauss}\\
\cmidrule(lr){1-4}
\multirow{3}{*}{$s\le 1/2$} & $\frac{d+\alpha}{d-2s}<q<\frac{2}{1-2s}$& Rubin + Ruiz as in \eqref {e-ext-smalls} &Weak Ni \eqref{e-weakNI}\\
\cmidrule(lr){2-4}
 & $q=\frac{2}{1-2s}$ & $L^q$-estimate \eqref{boundSupBall}& Weak Ni \eqref{e-weakNI}\\
 \cmidrule(lr){2-4}
 & $q>\frac{2}{1-2s}$& $L^q$-estimate \eqref{boundSupBall}& Rubin + Ruiz as in \eqref {e-ext-smalls}\\
\bottomrule
 \end{tabular}
\end{center}
\caption{Different cases in the proof of Proposition~\ref{prop-radial-int}}\label{table1}
\end{table}

\smallskip\noindent
\emph{Case $s> 1/2$.}
In the exterior of the ball $B_R(0)$, for any $p>\frac{2d}{d-2s}$ we can estimate
\begin{equation}\label{e-ext}
\int_{\R^d\setminus B_R(0)} \vert \varphi\vert^p\le
CR^{d-p\left(\frac{d}{2}-s\right)}\Vert \varphi\Vert _{\dot{H}^{s}(\R^d)}^{p},
\end{equation}
using the classical Sobolev inequality and Cho--Ozawa's inequality \eqref{Strauss}.
To obtain an estimate in the interior of the ball $B_R(0)$,
we observe that for $s>1/2$ we have $q<p_{\mathrm{rad}}$
and hence we can assume that $q<p<p_{\mathrm{rad}}$.
For a small $\eps>0$, set $\gamma:=\frac{d-\alpha}{2}-\eps$.
Then the estimate on $\int_{B_R(0)} \vert \varphi\vert^p$  is identical to the argument in \eqref{e-int-larges},
but carried out in the interior of the ball $B_R(0)$,
which reverses the inequalities in \eqref{e-deNap-prad} and \eqref{e-deNap-q}.

\smallskip\noindent
\emph{Case $s\le 1/2$ and $\frac{d+\alpha}{d-2s}<q<\frac{2}{1-2s}$.}
In the exterior of the ball $B_R(0)$ the estimate \eqref{e-ext}
follows directly from the weak Ni's inequality \eqref{e-weakNI}.
To obtain an estimate in the interior of the ball $B_R(0)$, observe that for $q<\frac{2}{1-2s}$ we have $q<p_{\mathrm{rad}}$
and hence we can assume that $q<p<p_{\mathrm{rad}}$.
For a small $\eps>0$, set
$\gamma:=\frac{d-\alpha}{2}-\eps$.
Then the estimate on $\int_{B_R(0)} \vert \varphi\vert^p$  is identical to the argument in \eqref{e-ext-smalls},
but carried out in the interior of the ball $B_R(0)$ with $q<p<r$.
The only difference is that for $q>\frac{d+\alpha}{d-2s}$ the inequality in \eqref{e-pgamma}
reverses and that $p_\eps\nearrow p_{\mathrm{rad}}$ as $\eps\to 0$,
since $\gamma<\frac{d-\alpha}{2}$.

Note that for $0<s<1/2$ and $q\ge\frac{2}{1-2s}$ we have $q\ge p_{\mathrm{rad}}$
and a H\"older inequality estimate
of type \eqref{e-ext-smalls} on $\int_{B_R(0)} \vert \varphi\vert^p$  is no longer possible.

\smallskip
\noindent
\emph{Case $s< 1/2$ and $q=\frac{2}{1-2s}$.}
Observe that in this case we have $p_{\mathrm{rad}}=q$.
In the exterior of the ball $B_R(0)$ the estimate
\begin{equation}\label{e-ext-q}
\int_{\R^d\setminus B_R(0)} \vert \varphi\vert^q\le
CR^{d-q\left(\frac{d}{2}-s\right)}\Vert \varphi\Vert _{\dot{H}^{s}(\R^d)}^{q},
\end{equation}
follows directly from the weak Ni's inequality \eqref{e-weakNI}, which is valid for $q=\frac{2}{1-2s}$.
To estimate $\int_{B_R(0)} \vert \varphi\vert^q$, we can use the $L^q$--estimate \eqref{boundSupBall}, i.e.
\begin{equation}
\label{boundSupBalllargeq-2}
\int_{B_R(0)} \abs{\varphi}^q
 \le CR^\frac{d - \alpha}{2}  \biggl(\ \iint_{\R^d\times \R^d}
    \frac{\abs{\varphi (x)}^q\,\abs{\varphi (y)}^q}{\abs{x - y}^{d-\alpha}}
    \dif x \dif y\biggr)^\frac{1}{2}.
\end{equation}
Combining \eqref{e-ext-q} and \eqref{boundSupBalllargeq-2} together we conclude that
$\mathcal{E}^{s, \alpha, q}_{\mathrm{rad}}(\R^d)\hookrightarrow L^{q}(\R^d)$,
the remaining range of $p$ follows by interpolation.

\smallskip
\noindent
\emph{Case $s< 1/2$ and $q>\frac{2}{1-2s}$.}
Observe that in this case $p<p_{\mathrm{rad}}<q$.
To estimate $\int_{B_R(0)} \vert \varphi\vert^p$, we use the $L^q$--estimate \eqref{boundSupBall} to obtain
\begin{equation*}
\label{boundSupBalllargeq-3}
\int_{B_R(0)} \abs{\varphi}^p
\le CR^{\big(1-\frac{p}{q}\big)\frac{d - \alpha}{2}}  \left(\iint_{\R^d\times \R^d}
\frac{\abs{\varphi (x)}^q\,\abs{\varphi (y)}^q}{\abs{x - y}^{d-\alpha}} \dif x \dif y\right)^{\frac{p}{2q}}.
\end{equation*}
To obtain an estimate in the exteriour of the ball $B_R(0)$,
we will use H\"older, Rubin and Ruiz's inequalities similarly to \eqref{e-ext-smalls},
with $\gamma=\frac{d-\alpha}{2}+\eps$ and $r<p<q$,
which could be carried out for $p<p_{\mathrm{rad}}$ provided that $\eps>0$ is sufficiently small, because $p_{\mathrm{rad}}>\frac{2}{1-2s}$.
\end{proof}

\begin{proof}[Proof of Theorem~\ref{thm:radialpartial}]
The scaling invariant inequalities of Theorem~\ref{thm:radialpartial} follow
from Propositions~\ref{prop-radial-ext} and~\ref{prop-radial-int}
by by the same scaling consideration as in the proof of Theorem~\ref{maint}.
\end{proof}

The estimates of Propositions~\ref{prop-radial-ext} and~\ref{prop-radial-int}
improve upon the estimate of Theorem~\ref{lem:gen} only when $\alpha>1$.
In the next section we show that the intervals of Propositions~\ref{prop-radial-ext} and~\ref{prop-radial-int}
are optimal and that for $\alpha\le 1$ there is no improvement for the radial embedding.

\section{Optimality of the radial embeddings}\label{s-radial-opt}
The optimality of the intervals in Theorems~\ref{maint} and~\ref{thm:radialpartial} for $s\le 1$ is a consequence of the following.

\begin{thm}\label{thm:sharp-q}
Let $d\geq 2$, $1<\alpha<d$, $0<s<1/2$ and $q=\frac{2}{1-2s}$.
Then the space $\mathcal{E}^{s, \alpha, q}_{\mathrm{rad}}(\R^d)$ is not continuously embedded into $L^{p}(\R^d)$ for $p>q=p_{\mathrm{rad}}$.
\end{thm}

\begin{thm}\label{thm:sharp}
Let $d\geq 2$, $1<\alpha<d$, $0<s \leq 1$ and $p,q\in[1,+\infty)$.
Then the space $\mathcal{E}^{s, \alpha, q}_{\mathrm{rad}}(\R^d)$ is not continuously embedded in
$L^{p}(\R^d)$ for
\begin{align}
p\le p_{\mathrm{rad}} & \quad\text{and}\quad \frac{1}{q}>\frac{d-2s}{d+\alpha},\\
p\ge p_{\mathrm{rad}} & \quad\text{and}\quad \frac{1}{q}<\frac{d-2s}{d+\alpha}, \; \frac{1}{q}\ne\frac{1-2s}{2}.
\end{align}
\end{thm}

\begin{thm}\label{thm:sharp2}
Let $d\geq 2$, $0<\alpha \leq 1$, $0<s \leq 1$ and $p,q\in[1,+\infty)$.
Then the space $\mathcal{E}^{s, \alpha, q}_{\mathrm{rad}}(\R^d)$ is not continuously embedded in $L^{p}(\R^d)$ for
\begin{align}
p<\frac{2 (2 q s+\alpha)}{2s+\alpha} & \quad\text{and}\quad \frac{1}{q}>\frac{d-2s}{d+\alpha},\\
p>\frac{2 (2 q s+\alpha)}{2s+\alpha} & \quad\text{and}\quad \frac{1}{q}<\frac{d-2s}{d+\alpha}.
\end{align}
\end{thm}

The proof of Theorems~\ref{thm:sharp} and~\ref{thm:sharp2} is obtained by constructing counterexamples, i.e
a family of functions $u$ such that for a suitable $p$ it holds
\begin{gather*}
\nonumber \Vert u\Vert _{\dot H^{s}(\R^d)}^{2}  \simeq 1\\
\label{eq:claim}\iint_{\R^d\times \R^d}\frac{\abs{u (x)}^q\,\abs{u (y)}^{q}}{\abs{x - y}^{d-\alpha}} \dif x \dif y \simeq  1\\
\nonumber \vert \abs{u}\vert _{L^p(\R^d)}^{p}  \rightarrow +\infty.
\end{gather*}
Given a nonnegative function \(\eta \in C^\infty (\R) \setminus \{0\}\) such that \(\operatorname{supp} \eta
\subset [-1, 1]\), we consider the family of functions
\begin{equation}\label{eq:u}
u_{\lambda,R,S}(x)= \lambda \, \eta \Bigl(\frac{\abs{x} - R}{S} \Bigr),
\end{equation}
where $R>S>0$ and $\lambda>0$ will be specified in the sequel.

By elementary computation we obtain
\begin{equation}\label{eq:compP}
\norm{u_{\lambda,R,S}}_{p}^p \simeq  \lambda^p R^{d-1}S.
\end{equation}
We also claim that
\begin{equation}
 \Vert u_{\lambda,R,S}\Vert _{\dot H^{s}(\R^d)}^{2}  \simeq  \lambda^2R^{d-1}S^{1-2s}, \label{eq:stimsob}
\end{equation}
and
\begin{equation}
 \iint_{\R^d\times \R^d}\frac{\abs{u_{\lambda,R,S} (x)}^q\,\abs{u_{\lambda,R,S} (y)}^{q}}{\abs{x - y}^{d-\alpha}} \dif x \dif y \lesssim
 \left\{\begin{array}{ll}
  \lambda^{2q}R^{d+\alpha-2}S^2  &\text{if $1<\alpha<d$},\smallskip\\
  \lambda^{2q}R^{d-1}S^2\log(R/S)  &\text{if $\alpha=1$},\smallskip\\
  \lambda^{2q}R^{d-1}S^{1+\alpha} &\text{if $0<\alpha<1$}.
 \end{array}\right. \label{eq:stimcoul}
\end{equation}
The estimate \eqref{eq:stimcoul} is proved in Appendix~\ref{AA} below.

To prove \eqref{eq:stimsob}, for any $s>0$ choose $k\in\N$ such that $2k\ge s$.
Taking into account that $S<R$, by the change of variables and scaling we compute
\begin{align}\label{eq:stimsob1-app2k}
  \Vert u_{\lambda,R,S}\Vert _{\dot H^{2k}(\R^d)}^{2}
  &\simeq \int_{\R^d}|\Delta^k u_{\lambda,R,S}|^2\dif x
  \simeq\int_0^\infty\left|\Big\{\frac{\partial^2}{\partial r^2}+\frac{d-1}{r}\frac{\partial}{\partial r}\Big\}^ku_{\lambda,R,S}(r)\right|^2r^{d-1}\dif r\\
  &=\int_0^\infty\left|\Big\{\frac{\partial^{2k}}{\partial r^{2k}}+\frac{a_1}{r}\frac{\partial^{2k-1}}{\partial r^{2k-1}}+\dots+\frac{a_k}{r^k}\frac{\partial^{k}}{\partial r^{k}}\Big\}u_{\lambda,R,S}(r)\right|^2 r^{d-1}dr\nonumber\\
  &\le \lambda^2d\left(\int_0^\infty\Big|\eta^{(2k)} \Bigl(\tfrac{r - R}{S} \Bigr)\Big|^2r^{d-1}dr
 +|a_1|\int_0^\infty\eta^{(2k-1)} \Big|\Bigl(\tfrac{r - R}{S} \Bigr)\Big|^2r^{d-3}dr\right.\nonumber\\
 &\hspace{10em}\left.+\dots
 +|a_d|\int_0^\infty\eta^{(k)} \Big|\Bigl(\tfrac{r - R}{S} \Bigr)\Big|^2 r^{d-1-2k}dr\right)\nonumber\\
 &\lesssim  \lambda^2 \Big(S^{1-4k}R^{d-1}+ S^{1-2(2k-1)}R^{d-3}+\dots+S^{1-2k}R^{d-1-k}\Big)\nonumber\\
 &\lesssim\lambda^2 S^{1-4k}R^{d-1}.\nonumber
\end{align}
Interpolating between the $L^2$ and $\dot{H}^{2k}$--norm of $u_{\lambda,R,S}$ (cf. \cite{BCD}*{Proposition 1.32}), we conclude
from \eqref{eq:compP} and \eqref{eq:stimsob1-app2k} that
\begin{equation*}
\Vert u_{\lambda,R,S}\Vert _{\dot H^{s}(\R^d)}^{2}\le \Vert u_{\lambda,R,S}\Vert _{\dot H^{2k}(\R^d)}^{\frac{s}{k}}\Vert u_{\lambda,R,S}\Vert _{L^2(\R^d)}^{2-\frac{s}{k}}\lesssim \lambda^2R^{d-1}S^{1-2s}.
\end{equation*}

\begin{proof}[Proof of Theorem~\ref{thm:sharp-q}]
Let $u_S:=u_{\lambda,R,S}$ be the function in \eqref{eq:u}, where we fix $R>0$ and for $S<R$ set
\begin{equation*}\label{eR-q}
\lambda=S^{-\frac{1}{q}}.
\end{equation*}
Then, since by our assumption \(1 < \alpha < d\),
\begin{gather}
 \Vert u_S\Vert _{\dot H^{s}(\R^d)}^{2}  \lesssim  R^{d-1}, \label{eq:stimsob-R-q}\\
 \iint_{\R^d\times \R^d}\frac{\vert u_S(x)\vert^q\,\vert u_S(y)\vert^{q}}{\abs{x - y}^{d-\alpha}} \dif x \dif y \lesssim  R^{d+\alpha-2}   \label{eq:stimcoul-R-q},\\
 \Vert u_S\Vert _{L^p(\R^d)}^{p}\simeq\lambda^{p} S R^{d-1} \simeq \lambda^{p-q} R^{d-1} \simeq
 S^{1-\frac{p}{q}}R^{d-1},
 \label{eq:stimq-R-q}
\end{gather}
Since $R$ is fixed, we conclude that $\Vert u_S\Vert _{L^p(\R^d)}\to\infty$ for $p>q$ when $S\to 0$.
\end{proof}

\begin{proof}[Proof of Theorem~\ref{thm:sharp}]
Let $u_R:=_{\lambda,R,S}$ be the function in \eqref{eq:u}, where we set
\begin{equation*}\label{eR}
\lambda=R^{\beta}\quad\text{and}\quad
S=\bigl(\lambda^2R^{d-1}\bigr)^{\frac{1}{2s-1}}=R^\gamma,
\end{equation*}
with
\begin{equation}\label{eRtgamma}
\beta=-\frac{2(d-1)+(d+\alpha-2)(2s-1)}{2q (2s-1)+4},\qquad\gamma=\frac{q(d-1)-(d+\alpha-2)}{q(2s-1)+2}.
\end{equation}
Then we compute
\begin{gather}
 \Vert u_R\Vert _{\dot H^{s}(\R^d)}^{2}  \lesssim  1, \label{eq:stimsob-R}\\
 \iint_{\R^d\times \R^d}\frac{\vert u_R(x)\vert^q\, \vert u_R(y)\vert^{q}}{\abs{x - y}^{d-\alpha}} \dif x \dif y \lesssim  1   \label{eq:stimcoul-R},\\
 \Vert u_R\Vert _{L^p(\R^d)}^{p}\simeq\lambda^{p}R^{d-1}S
 \simeq R^{\beta (p - p_{\mathrm{rad}})},
 \label{eq:stimq-R}
\end{gather}
provided that $R>S$, that is, either \(R > 1\) and \(\gamma < 1\) or \(R < 1\) and \(\gamma > 1\).
To complete the proof Theorem~\ref{thm:sharp} for $p\neq p_{\mathrm{rad}}$ we select $R$
according to Table~\ref{table2}.

\begin{table}[h!]
\begin{center}
\begin{tabular}{ c c  c c c }
\toprule
$q$ &  $\beta$ & $\gamma$ & Choice of $R$& Conclusion \\
\midrule
$\frac{1}{q}>\frac{d-2s}{\alpha+d}$ &  $\beta<0$ & $0<\gamma<1$ & $R\to\infty$ & $\Vert u_R\Vert _{L^p(\R^d)}^{p}\to\infty$ for $p<p_{\mathrm{rad}}$\\
\addlinespace
$\frac{1}{q}\in\left(\big(\frac{1-2s}{2}\big)_+,\frac{d-2s}{\alpha+d}\right)$ & $\beta<0$ & $\gamma>1$& $R\to 0$ & $\Vert u_R\Vert _{L^p(\R^d)}^{p}\to\infty$ for $p>p_{\mathrm{rad}}$\\
\addlinespace
$s<1/2$ and $\frac{1}{q}<\frac{1-2s}{2}$  & $\beta>0$ & $\gamma<0$& $R\to\infty$ & $\Vert u_R\Vert _{L^p(\R^d)}^{p}\to\infty$ for $p>p_{\mathrm{rad}}$\\
\bottomrule
\end{tabular}
\end{center}
\caption{Choice of $R$ which ensures $R>S$ and $\Vert u_R\Vert _{L^p(\R^d)}^{p}\to\infty$ for $\alpha>1$.}\label{table2}
\end{table}

\smallskip
Next we prove that $\mathcal{E}^{s, \alpha, q}_{\mathrm{rad}}(\R^d)\not\subset L^{p_{\mathrm{rad}}}(\R^d)$ when $\frac{1}{q}\neq\frac{1-2s}{2}$.
Similarly to \cite{MMVS}*{Lemma 6.4}, we consider the ``multibump'' sequence
\begin{equation*}
v_{R,m}=\sum_{k=1}^m u_{R^k},
\end{equation*}
where the functions $u_{R^k}$ are as in \eqref{eq:u} with $R=R^k$, $\lambda=R^{k\beta}$, $S=R^{k \frac{2 \beta + d - 1}{2 s - 1}}$ and where $\beta$ is given in \eqref{eRtgamma}.
Note that for $R\neq 1$ and sufficiently large quotient $R/S$ the functions $u_{R^k}$ ($k=1,\dots,m$) have mutually disjoint supports.

If $\frac{1}{q}>\frac{d-2s}{\alpha+d}$, or $s<1/2$ and $\frac{1}{q}<\frac{1-2s}{2}$ then we let $R\to\infty$.
We obtain
\begin{gather}
 \Vert v_{R,m}\Vert _{L^p(\R^d)}^{p}\simeq m,\label{eq:stimq-m}\\
 \Vert v_{R,m}\Vert _{\dot H^{s}(\R^d)}^{2}  \lesssim  m, \label{eq:stimsob-m}\\
 \iint_{\R^d\times \R^d}\frac{\vert v_{R,m}(x)\vert^q\,\vert v_{R,m}(y)\vert^{q}}{\abs{x - y}^{d-\alpha}} \dif x \dif y \lesssim  m   \label{eq:stimcoul-m}.
\end{gather}
For derivation of \eqref{eq:stimcoul-m} see \cite{MMVS}*{proof of Lemma 6.4}.
To obtain \eqref{eq:stimsob-m}, we
observe that
\begin{equation}\label{e-summm}
\Vert v_{R,m}\Vert _{\dot H^{s}(\R^d)}^{2}=\sum_{k=1}^m\Vert u_{R^k}\Vert _{\dot H^{s}(\R^d)}^{2}+2\sum_{i,j=1,i>j}^m( u_{R^i},u_{R^j})_{\dot H^{s}(\R^d)}.
\end{equation}
If \(s\) is an integer the second term vanishes, or if $s<1$ then the second term is negative.
Otherwise, \(s = \ell + \sigma\), with $\ell\in\N$ and \(\sigma \in (0, 1)\). Thus by the Gagliardo seminorm characterization of \(\dot H^s (\R^d)\), if \(u_{R^i}\) and \(u_{R^j}\) have disjoint supports,
\begin{equation}\label{eq:ell1}
\begin{split}
  ( u_{R^i},u_{R^j})_{\dot H^{s}(\R^d)}
  &= \iint_{\R^d\times \R^d}\frac{(\nabla^\ell u_{R^i} (x) - \nabla^\ell u_{R^i} (y))\cdot (\nabla^\ell u_{R^j} (x) - \nabla^\ell u_{R^j} (y)}{\abs{x - y}^{d + 2 \sigma}}\dif x \dif y\\
  &= -2 C \iint_{\R^d\times \R^d}\frac{\nabla^\ell u_{R^i} (x) \cdot \nabla^\ell u_{R^j} (y)}{\abs{x - y}^{d + 2 \sigma}}\dif x \dif y.
\end{split}
\end{equation}
Similarly to \eqref{eq:stimsob1-app2k}, we deduce that $\norm{D^\ell u_{\lambda,R,S}}_{L^1(\R^d)}\lesssim \lambda R^{d-1}S^{1-\ell}$ and hence
\begin{equation}\label{eq:l1}
\norm{D^\ell u_{R^k}}_{L^1(\R^d)}\lesssim R^{k(\beta+d-1+\gamma(1-\ell))}.
\end{equation}
If $\frac{1}{q}>\frac{d-2s}{\alpha+d}$ then $\beta<0$ and $0<\gamma<1$. For $i>j$ and if \(R^i \gg R^j\) we estimate \eqref{eq:ell1} as follows,
\begin{equation}\label{eq:ell3}
\begin{split}
  ( u_{R^i},u_{R^j})_{\dot H^{s}(\R^d)}
  &\lesssim
  \frac{\norm{D^\ell u_{R^i}}_{L^1(\R^d)}\norm{D^\ell u_{R^j}}_{L^1(\R^d)} }{\big(R^i-R^j\big)^{d + 2\sigma}}\\
  &\lesssim
  R^{-i(d + 2\sigma)}R^{(i + j) (\beta + d - 1 +\gamma(1-\ell))}\\
  &\lesssim R^{-i(d + 2\sigma)}R^{i(2(\gamma s-\beta) + 2\sigma\gamma)}\lesssim R^{-i(2\sigma(1-\gamma))},
\end{split}
\end{equation}
since we note that $2(\gamma s-\beta)<d$, provided that $q<\frac{d+\alpha}{d-2s}$.
Then in \eqref{e-summm} for all sufficiently large $R$ we have
\begin{equation}
\Vert v_{R,m}\Vert _{\dot H^{s}(\R^d)}^{2}\lesssim m+\sum_{i,j=1,i>j}^m R^{-i(2\sigma(1-\gamma))}\lesssim m.
\end{equation}
The case $\frac{1}{q}\in\left(\big(\frac{1-2s}{2}\big)_+,\frac{d-2s}{\alpha+d}\right)$ is similar, but letting $R\to 0$
and observing that $\gamma<0$.

Now, set
\begin{equation*}
w_{R,m}(x)=m^{\theta} v_{R,m}\Big(\frac{x}{m^{\sigma}}\Big).
\end{equation*}
Then by the standard scaling we have
\begin{gather}
 \Vert w_{R,m}\Vert _{L^p(\R^d)}^{p}\simeq m^{p\theta+\sigma d+1},\label{eq:stimq-m-1}\\
 \Vert w_{R,m}\Vert _{\dot H^{s}(\R^d)}^{2}  \lesssim  m^{2\theta+\sigma(d-2s)+1}, \label{eq:stimsob-m-1}\\
 \iint_{\R^d\times \R^d}\frac{\vert w_{R,m}(x)\vert^q\,\vert w_{R,m}(y)\vert^{q}}{\abs{x - y}^{d-\alpha}} \dif x \dif y \lesssim  m^{2q\theta+\sigma(d+\alpha)+1}  \label{eq:stimcoul-m-1}.
\end{gather}
If we set
$$\sigma=\frac{q-1}{d+\alpha-q(d-2s)},\qquad\theta=-\frac{2s+\alpha}{2(d+\alpha-q(d-2s))},$$
then for $R\to\infty$ and $m\to\infty$ we obtain
\begin{align}
 &\Vert w_{R,m}\Vert _{\dot H^{s}(\R^d)}^{2}  \lesssim 1, \label{eq:stimsob-mbg}\\
 &\iint_{\R^d\times \R^d}\frac{\vert w_{R,m}(x)\vert^q\,\vert w_{R,m}(y)\vert^{q}}{\abs{x - y}^{d-\alpha}} \dif x \dif y \lesssim  1  \label{eq:stimcoul-mbg},\\
 &\Vert w_{R,m}\Vert _{L^p(\R^d)}^{p}\simeq m^{p\theta+\sigma d+1}\simeq m^\frac{2s(\alpha-1)}{2s(d+\alpha-2)+d-\alpha}\to\infty,\label{eq:stimq-mbg}
\end{align}
since $\alpha>1$ and $d\ge 2$.

The case $\frac{1}{q}\in\left(\big(\frac{1-2s}{2}\big)_+,\frac{d-2s}{\alpha+d}\right)$ is similar, by letting $R\to 0$.
\end{proof}

\begin{proof}[Proof of Theorem~\ref{thm:sharp2}]
The strategy in the case $0<\alpha<1$ and $\frac{1}{q}\neq \frac{1-2s}{1+\alpha}$ is the same as in the first part of the proof of Theorem~\ref{thm:sharp}.
Let $u_R:=u_{\lambda,R,S}$ be the function in \eqref{eq:u} and we choose
\begin{equation*}\label{eR1}
\lambda=R^{\beta},\qquad S=\bigl(\lambda^2R^{d-1}\bigr)^{\frac{1}{2s-1}}=R^\gamma,
\end{equation*}
where
\begin{equation*}\label{eR1betagamma}
\beta=-\frac{(d-1)(2s+\alpha)}{2(q (2s-1)+1+\alpha)},\qquad \gamma=\frac{(d-1)(q-1)}{q(2s-1)+1+\alpha}.
\end{equation*}
Then \eqref{eq:stimsob-R} and \eqref{eq:stimcoul-R} hold, and
\[
\norm{u_R}_{L^p(\R^d)}^{p}\simeq\lambda^{p}R^{d-1}S \simeq R^{\beta (p-\frac{2 (2qs+\alpha)}{2s+\alpha})},
\]
provided that $R>S$.
Then to construct the required counterexamples, we select $R$ according to Table~\ref{table3}.

\begin{table}[h!]
\begin{center}
\begin{tabular}{ c c  c c c }
\toprule
$q$  & $\beta$ & $\gamma$ & Choice of $R$& Conclusion \\
\midrule
$\frac{1}{q}>\frac{d-2s}{\alpha+d}$ & $\beta<0$& $0<\gamma<1$  & $R\to\infty$ & $\Vert u_R\Vert _{L^p(\R^d)}^{p}\to\infty$ for $p<\frac{2 (2qs+\alpha)}{2s+\alpha}$\\
\addlinespace
$\frac{1}{q}\in\left(\big(\frac{1-2s}{1+\alpha}\big)_+,\frac{d-2s}{\alpha+d}\right)$ & $\beta<0$ & $\gamma>1$& $R\to 0$ & $\Vert u_R\Vert _{L^p(\R^d)}^{p}\to\infty$ for $p>\frac{2 (2qs+\alpha)}{2s+\alpha}$\\
\addlinespace
$s<1/2$ and $\frac{1}{q}<\frac{1-2s}{1+\alpha}$  & $\beta>0$ & $\gamma<0$ & $R\to\infty$ & $\Vert u_R\Vert _{L^p(\R^d)}^{p}\to\infty$ for $p>\frac{2(2qs+\alpha)}{2s+\alpha}$\\
\bottomrule
\end{tabular}
\end{center}
\caption{Choice of $R$ which ensures $R>S$ and $\Vert u_R\Vert _{L^p(\R^d)}^{p}\to\infty$ for $\alpha \le 1$.}\label{table3}
\end{table}

In the case $0<\alpha<1$, $s<1/2$ and $q= \frac{1+\alpha}{1-2s}$ we note that $\frac{2(2qs+\alpha)}{2s+\alpha}=\frac{2}{1-2s}>q$.
Similarly to the proof of Theorem~\ref{thm:sharp-q}, for $u_S:=u_{\lambda,R,S}$ with a fixed $R>0$ and for $S<R$ we set
\begin{equation*}\label{eR-q1}
\lambda=S^{-\frac{\alpha+2s}{2(q-1)}}=S^{\frac{2s-1}{2}}.
\end{equation*}
Then
\begin{gather}
 \Vert u_S\Vert _{\dot H^{s}(\R^d)}^{2}  \simeq  R^{d-1}, \label{eq:stimsob-R-q1}\\
 \iint_{\R^d\times \R^d}\frac{\vert u_S(x)\vert^q\,\vert u_S(y)\vert^{q}}{\abs{x - y}^{d-\alpha}} \dif x \dif y \lesssim  R^{d-1}   \label{eq:stimcoul-R-q1},\\
 \Vert u_S\Vert _{L^p(\R^d)}^{p}\simeq\lambda^{p} S R^{d-1} \simeq \lambda^{p-\frac{2}{2s-1}} R^{d-1} \simeq
 S^{1-\frac{p(1-2s)}{2}}R^{d-1}.
 \label{eq:stimq-R-q1}
\end{gather}
Since $R$ is fixed, we conclude that $\Vert u_S\Vert _{L^p(\R^d)}\to\infty$ for $p>\frac{2(2qs+\alpha)}{2s+\alpha}=\frac{2}{1-2s}$ when $S\to 0$.

The case $\alpha=1$ is similar, but takes into account the logarithmic correction in \eqref{eq:stimcoul}. We omit the details.
\end{proof}

\section{Radial compactness: Proof of Theorem~\ref{thm:radial comp}}\label{compactness section}

We need the following preliminary local compactness result.

\begin{lem}[Local compactness]\label{local compactness}
Let $d\in \mathbb N$, $s>0$, $\alpha\in (0,d)$ and $q\in [1,\infty)$. Then the embedding $\mathcal{E}^{s, \alpha, q} (\R^d)\hookrightarrow L^{1}_{\textrm{loc}}(\R^d)$ is compact.
\end{lem}
\begin{proof}
Multiplication by $\theta\in \mathcal S(\R^d)$ is a continuous mapping $\mathcal{E}^{s, \alpha, q} (\R^d)\rightarrow \dot{H}^s(\R^d)$.
Indeed by the fractional Leibniz rule, see e.g. \cite{Gulisashvili}*{Theorem 1.4}, we obtain

$$\Vert (-\Delta )^{\frac{s}{2}} \theta u \Vert _{L^2(\R^d)}\lesssim\Vert (-\Delta )^{\frac{s}{2}}  u \Vert _{L^2(\R^d)} \Vert  \theta  \Vert _{L^\infty(\R^d)}+ \Vert (-\Delta )^{\frac{s}{2}} \theta \Vert _{L^{r}(\R^d)}\Vert u \Vert _{L^{\frac{2 (2 q s+\alpha)}{2s+\alpha}}(\R^d)},$$with $r$ such that $\frac{2s+\alpha}{2 (2 q s+\alpha)}+\frac{1}{r}=\frac{1}{2}$. For $q=1, $ we set $r=\infty$. Hence by Theorem~\ref {lem:gen},
\begin{equation*}\label{contmapp}
\Vert \theta u \Vert _{\dot{H}^s(\R^d)} \leq C(\theta) \Vert u \Vert _{\mathcal{E}^{s, \alpha, q} (\R^d)}.
\end{equation*}
For every \(\rho > 0\),
we choose \(\theta \in C^\infty (\R^d)\) such that \(\theta = 1\) on \(B_\rho\) and \(\theta = 0\) in \(\R^d \setminus B_{2 \rho}\).
Let $(u_n)_{n\in\mathbb N}\) be a bounded sequence in $\mathcal{E}^{s, \alpha, q} (\R^d)$.
Setting $v_n=\theta u_n$, theorem~\ref {lem:gen}  implies that $(v_n)_{n\in\mathbb N}$ is also bounded in $H^{s} (\R^d)$. We can assume that $v_n$ converges weakly to some $v$ in $L^2(\R^d)$.
By testing against suitable test functions, it follows that $v$ is also supported in $B_{2 \rho}$ and thus $\Hat v\in L^\infty(\R^d)$.
By Plancharel's identity we have
\begin{equation*}
\Vert v_n-v\Vert _{L^{2}(\R^d)}^2=\int_{\vert \xi\vert \leq R} \vert \widehat v_n(\xi)-\widehat v(\xi)\vert^2\dif \xi+\int_{\vert \xi\vert > R}\vert \widehat v_n(\xi)-\widehat v(\xi)\vert^2\dif \xi.
\end{equation*}

By showing that the right hand side goes to zero we will infer by H\"older's inequality that $\Vert u_n-v\Vert _{L^{1}(B_\rho)}\rightarrow 0$. We have
$$\int_{\vert \xi\vert > R} \vert \widehat v_n(\xi)-\widehat v(\xi)\vert^2\dif \xi \leq \frac{1}{(1+R^2)^{s}}\int_{\R^d} \bigl(1+\vert \xi\vert^2\bigr)^{s}\vert \widehat v_n(\xi)-\widehat v(\xi)\vert^2\dif \xi\leq  \frac{C}{(1+R^2)^{s}}.$$
Since $e^{ix\cdot \xi}\in L^2_x(B_{2 \rho})$, by weak convergence in $L^2(B_{2 \rho})$ we have $\widehat v_n(\xi) \rightarrow \widehat v (\xi)$ almost everywhere. To conclude it suffices to show that
\begin{equation}\label{ts}
\int_{\vert \xi\vert \leq  R} \vert \widehat v_n(\xi)-\widehat v(\xi)\vert^2\dif \xi =o(1).
\end{equation}
Notice that $\Vert \widehat v_n \Vert _{\infty}\leq \Vert v_n\Vert _{L^1(B_{2 \rho})}\leq
\mu(B_{2 \rho})^{\frac 12} \Vert v_n\Vert _{L^2(B_{2 \rho})}\leq \mu(B_{2 \rho})^{\frac 12}\Vert v_n\Vert _{H^s(\R^d)}$ and hence  $\vert \widehat v_n(\xi)-\widehat v(\xi)\vert^2$ is estimated by a uniform constant so that by Lebesgue's dominated convergence theorem
\eqref{ts} holds. This concludes the proof.
\end{proof}

\begin{proof}[Proof of theorem~\ref{thm:radial comp} ]
We sketch the proof only in the most interesting case $\alpha>1$, $s<1/2$, and $q\geq \frac{2}{1-2s}$, namely when $p_{\mathrm{rad}}\leq q$. Notice that for all $R>0$, by \eqref{boundSupBall} and Lemma~\ref{local compactness}, interpolation between $q$ and $p'=1$ yields the compact embedding $\mathcal{E}^{s, \alpha, q}_{\mathrm{rad}}(\R^d) \hookrightarrow L^p_{loc}(\R^d)$ for all $1\leq p<q$. Thus it suffices to show that for any bounded sequence $(u_n)_{n\in\mathbb N}$ in $\mathcal{E}^{s, \alpha, q}_{\mathrm{rad}} (\R^d)$ it holds that $$\sup_{n\in \mathbb N} \int_{\R^d\setminus B_R(0)} \vert u_n\vert^p\rightarrow 0,\quad R\rightarrow \infty.$$
When $p\leq \frac{2}{1-2s}$, we use Lemma~\ref{cstep0} which yields
$$\int_{\R^d \setminus B_R(0)}\vert u_n\vert^p\leq o(1)\Vert u_n\Vert ^p_{\mathcal{E}^{s, \alpha, q}(\R^d)}, \qquad R\rightarrow \infty.$$
When $p>\frac{2}{1-2s}$ the same conclusion holds by arguing as in the proof of \eqref{e-ext-smalls} and using the strict inequality $p<p_{\textrm{rad}}$.
This is enough to prove the theorem for $\alpha>1$, $s<1/2$, and $q\geq \frac{2}{1-2s}$.

The other cases are similar, estimating the various integrals as in Proposition~\ref{prop-radial-ext} for $q<\frac{d+\alpha}{d-2s}$ and according to Table~\ref{table1} for $q>\frac{d+\alpha}{d-2s}$. This concludes the proof.
\end{proof}

\appendix
\section{{Proof of claim (\ref{eq:stimcoul}) }}\label{AA}

\begin{proof}[Proof of \eqref{eq:stimcoul}.]
We use an estimate for radial functions from \cite{MMVS}.
Similar estimates were previously obtained in \citelist{\cite{Rubin-1982}\cite{Duoandikoetxea13}\cite{Thim2016}}.

\begin{lem}[\cite{MMVS}*{Lemma 6.3}]
Let $d \geq 2$ and $\alpha\in (0,d)$,  then for every measurable function $f:[0,\infty) \rightarrow [0,\infty)$
$$\iint_{\R^d\times \R^d}
\frac{f(\abs{x}) f(\vert y\vert )}{\abs{x - y}^{d-\alpha}} \dif x \dif y=\int_0^{\infty}\int_0^{\infty}f( r )  K_{\alpha,d}^R(r,s)f(s)r^{d-1}s^{d-1}\dif r \dif s$$
where the kernel $K_{\alpha,d}^R: [0,\infty) \times [0,\infty) \rightarrow \infty$ is defined for $r,s \in  [0,\infty) \times [0,\infty)$ by
$$K_{\alpha,d}^R(r,s)=C_d \int_0^1\frac{z^{\frac{d-3}{2}}(1-z)^{\frac{d-3}{2}}}{((s+r)^2-4srz)^{\frac{d-\alpha}{2}}} \dif z. $$
Moreover, there exists $M>0$ such that
\begin{equation}\label{AS}
K_{\alpha,d}^R(r,s)\leq M
\left\{\begin{array}{ll}
(\frac{1}{rs})^{\frac{d-1}{2}}  \frac{1}{\vert r-s\vert^{1-\alpha}} & \text{ if } \alpha<1, \smallskip\\
(\frac{1}{rs})^{\frac{d-1}{2}}  \ln \frac{2\vert r+s\vert }{\vert r-s\vert } &\text{ if } \alpha=1, \smallskip  \\
(\frac{1}{rs})^{\frac{d-\alpha}{2}}  &\text{ if } \alpha>1.
\end{array}\right.
\end{equation}
\end{lem}

\medskip\noindent
\emph{Case $\alpha>1$.}
From \eqref{AS} we obtain  for radially symmetric functions that
$$\iint_{\R^d\times \R^d}
\frac{\abs{\varphi (x)}^q\,\abs{\varphi (y)}^q}{\abs{x - y}^{d-\alpha}} \dif x \dif y\leq C \int_0^{\infty}\int_0^{\infty}\
\frac{\vert \varphi( r ) \vert^q \, \vert \varphi( s ) \, \vert^q r^{d-1}s^{d-1}}{(r s)^{\frac{d-\alpha}{2}}} \dif r \dif s,$$
and hence that
$$\iint_{\R^d\times \R^d}
\frac{\abs{\varphi (x)}^q\,\abs{\varphi (y)}^q}{\abs{x - y}^{d-\alpha}} \dif x \dif y\leq C  \left(\int_0^{\infty}
\vert \varphi( r ) \vert^q \, r^{\frac{d}{2}+\frac{\alpha}{2}-1} \dif r \right)^{2}.$$
Let $u=u_{\lambda,R,S}$ be defined in \eqref{eq:u}. Then
$$\iint_{\R^d\times \R^d}
\frac{\abs{u (x)}^q\,  \abs{u (y)}^q}{\abs{x - y}^{d-\alpha}} \dif x \dif y\leq C  \lambda^{2q} \left(\int_{R-S}^{R+S}
\left(\frac{ S-\big\vert  r-R\big\vert }{S}\right)^q r^{\frac{d}{2}+\frac{\alpha}{2}-1} \dif r \right)^{2}.$$
Using the trivial estimate $\frac{ S-\big\vert  r-R\big\vert }{S}<1$ it follows that
$$\iint_{\R^d\times \R^d}
\frac{\abs{u (x)}^q\, \abs{u (y)}^q}{\abs{x - y}^{d-\alpha}} \dif x \dif y\leq C  \lambda^{2q} \left((R+S)^{\frac{d}{2}+\frac{\alpha}{2}}-(R-S)^{\frac{d}{2}+\frac{\alpha}{2}}\right)^2$$
and we get the desired estimate.

\medskip\noindent
\emph{Case $\alpha=1$.}
From \eqref{AS} we obtain  for radially symmetric functions that
$$\iint_{\R^d\times \R^d}
\frac{\abs{\varphi (x)}^q\,\abs{\varphi (y)}^q}{\abs{x - y}^{d-\alpha}} \dif x \dif y\leq C \int_0^{\infty}\int_0^{\infty}\
\frac{\vert \varphi( r ) \vert^q \vert \varphi( s ) \vert^q r^{d-1}s^{d-1}}{(r s)^{\frac{d-1}{2}}} \ln \frac{2\vert r+s\vert }{\vert r-s\vert }\dif r \dif s,$$
and hence that
$$\iint_{\R^d\times \R^d}
\frac{\abs{\varphi (x)}^q\,\abs{\varphi (y)}^q}{\abs{x - y}^{d-\alpha}} \dif x \dif y\leq C  \int_0^{\infty}\int_0^{\infty}
\vert \varphi( r ) \vert^q \vert \varphi( s ) \vert^q r^{\frac{d}{2}-\frac{1}{2}}s^{\frac{d}{2}-\frac{1}{2}} \ln \frac{2\vert r+s\vert }{\vert r-s\vert }\dif r \dif s.$$
Let $u=u_{\lambda,R,S}$ be defined in \eqref{eq:u}.
Using the  estimates $\frac{ S-\big\vert  r-R\big\vert }{S}<1$ and $r\leq R+S$, $s\leq R+S$ we have
$$\iint_{\R^d\times \R^d}
\frac{\abs{u (x)}^q\,\abs{u (y)}^q}{\abs{x - y}^{d-\alpha}} \dif x \dif y\leq C \lambda^{2q} (R+S)^{d-1}\int_{R-S}^{R+S}\int_{R-S}^{R+S} \ln \frac{2\vert r+s\vert }{\vert r-s\vert }\dif r \dif s$$
and we can conclude that
$$\iint_{\R^d\times \R^d}
\frac{\abs{u (x)}^q\,\abs{u (y)}^q}{\abs{x - y}^{d-\alpha}} \dif x \dif y\leq C \lambda^{2q} R^{d-1}\int_{R-S}^{R+S}\int_{R-S}^{R+S} \ln \frac{2\vert r+s\vert }{\vert r-s\vert }\dif r \dif s$$
i.e.
$$\iint_{\R^d\times \R^d}
\frac{\abs{u (x)}^q\,\abs{u (y)}^q}{\abs{x - y}^{d-\alpha}} \dif x \dif y\leq C \lambda^{2q} R^{d-1}S^2(\ln R-\ln S+1)=O(\lambda^{2q}R^{d-1+\beta}S^2).$$

\medskip\noindent
\emph{Case $0<\alpha<1$.}
This case is similar to $\alpha=1$, we omit the details.
\end{proof}

\section*{Acknowledgements}

J.~Bellazzini and M.~Ghimenti were supported by GNAMPA 2016 project ``Equazioni non lineari dispersive''.
M.~Ghimenti was partially supported by P.R.A. 2016, University of Pisa.
J.~Van Schaftingen was supported by the Projet de Recherche (Fonds de la Recherche Scientifique--FNRS) T.1110.14 ``Existence and asymptotic behavior of solutions to systems of semilinear elliptic partial differential equations''.


\end{document}